\documentclass[reqno,11pt]{amsart}
\usepackage{a4wide,xcolor,eucal,enumerate,mathrsfs}
\usepackage[normalem]{ulem}
\usepackage{pdfsync}
\usepackage[colorlinks, citecolor=blue]{hyperref}
\usepackage[capitalize]{cleveref}
\usepackage{amsmath,amssymb,epsfig,amsthm}
\usepackage[utf8]{inputenc}
\usepackage[english]{babel}
\usepackage{csquotes}
\numberwithin{equation}{section}

\newtheorem{theorem}{Theorem}[section]

\newtheorem{corollary}[theorem]{Corollary}
\newtheorem{lemma}[theorem]{Lemma}
\newtheorem{proposition}[theorem]{Proposition}
\newtheorem{remark}[theorem]{Remark}

\renewcommand{\H}{\mathbb{H}}

\newcommand{\N}{\mathbb{N}}

\newcommand{\R}{\mathbb{R}}

\newcommand{\cB}{{\ensuremath{\mathcal B}}}

\newcommand{\cL}{{\ensuremath{\mathcal L}}}

\newcommand{\cR}{{\ensuremath{\mathcal R}}}

\newcommand{\mm}{{\mbox{\boldmath$m$}}}

\newcommand{\sfb}{{\sf b}}

\newcommand{\sfd}{{\sf d}}

\newcommand{\sfg}{{\sf g}}
\newcommand{\sfh}{{\sf h}}

\newcommand{\sfn}{{\sf n}}

\newcommand{\sfs}{{\sf s}}

\newcommand{\sfH}{{\sf H}}

\newcommand{\sfM}{{\sf M}}
\newcommand{\sfN}{{\sf N}}

\newcommand{\sfP}{{\sf P}}

\newcommand{\sfR}{{\sf R}}

\newcommand{\sfX}{{\sf X}}

\newcommand{\supp}{\mathop{\rm supp}\nolimits}   
\newcommand{\restr}[1]{\lower3pt\hbox{$|_{#1}$}}

\newcommand{\Leb}[1]{{\mathscr L}^{#1}}      

\newcommand{\eps}{\varepsilon}  
\newcommand{\nchi}{{\raise.3ex\hbox{$\chi$}}}

\def\qed{\ifmmode 
  \else \leavevmode\unskip\penalty9999 \hbox{}\nobreak\hfill
  \fi               
    \qquad           \hbox{\hskip.5em $\square$
                \hskip.1em}}

\renewcommand{\mm}{\mathfrak m} 
\newcommand{\Probabilities}[1]{\mathscr P(#1)}          

\newcommand{\CD}{{\sf CD}}
\newcommand{\RCD}{{\sf RCD}}

\newcommand{\bra}[1]{\left( #1 \right)}
\newcommand{\sqa}[1]{\left[ #1 \right]}
\newcommand{\cur}[1]{\left\{ #1 \right\}}
\newcommand{\ang}[1]{\left< #1 \right>}
\newcommand{\abs}[1]{\left| #1 \right|}
\newcommand{\nor}[1]{\left\| #1 \right\|}

\newcommand{\veps}{\varepsilon}

\renewcommand{\d}{\mathsf{d}}
\newcommand{\Ric}{\mathbf{Ric}}

\renewcommand{\H}{\operatorname{H}}

\newcommand{\vecL}{\vec{\cL}}
\renewcommand{\veps}{\varepsilon}
\renewcommand{\Cap}{\operatorname{Cap}_2}
\newcommand{\Lip}{\operatorname{Lip}_{{\sf bs}}}

\newcommand{\test}[1]{{\rm Test}(#1)}

\newcommand{\testf}[1]{{\rm TestForm}(#1)}

\newcommand{\bd}{{\mathbf\Delta}}
\newcommand{\h}{{\sfh}}
\newcommand{\ddt}{\frac{\de}{\de t}}

\renewcommand{\cR}{\mathfrak{R}}
\newcommand{\de}{\mathrm{d}}
\newcommand{\bell}{\mathrm{B}}
\newcommand{\Meas}{\mathcal{M}}
\newcommand{\TV}{\mathsf{TV}}

\newcommand{\Id}{\operatorname{Id}}
\newcommand{\dist}{\mathbf{d}}

\usepackage[hyperref, url=false, doi=false,  isbn=false, giveninits=true, backref,style=numeric,maxbibnames=99, backend=biber]{biblatex}

\bibliography{riesz_biblio}

\title{Boundedness of Riesz transforms on $\RCD(K, \infty)$ spaces}

\author{Andrea Carbonaro}
\address{Andrea Carbonaro, Universit\`a degli Studi di Genova, Dipartimento di Matematica, \\ Via Dodecaneso 35 \\ I-16146 Genova, Italy}
\email{carbonaro@dima.unige.it}

\author{Luca Tamanini}
\address{Luca Tamanini, Universit\`a Cattolica del Sacro Cuore, Dipartimento di Matematica e Fisica `Niccol\`o Tartaglia' \\ Via della Garzetta 48 \\ I-25133 Brescia, Italy}
\email{luca.tamanini@unicatt.it}

\author{Dario Trevisan}
\address{Dario Trevisan, Universit\`a degli Studi di Pisa, Dipartimento di Matematica,  \\
Largo Bruno Pontecorvo 5 \\ I-56127 Pisa, Italy}
\email{dario.trevisan@unipi.it}

\subjclass[2020]{Primary 53C21. Secondary 42B20, 46E36.}
\keywords{Riesz transforms; Metric measure spaces; Bellman function.}
\date{\today}

\begin{document}

\maketitle

\begin{abstract}
For $1<p<\infty$, we prove the $L^p$-boundedness of the Riesz transform operators on metric measure spaces with Riemannian Ricci curvature bounded from below, without any restriction on their dimension. This large class of spaces include e.g.\ that of Hilbert spaces endowed with a log-concave probability measure. As a consequence, we extend the range of the validity of the Lusin-type approximation of Sobolev by Lipschitz functions, previously obtained by L.\ Ambrosio, E.\ Bru\`e and the third author in the quadratic case, i.e.\ $p=2$.  The proofs are analytic and rely on computations on an explicit Bellman function.
\end{abstract}







\setcounter{tocdepth}{1}
\tableofcontents

\section{Introduction}

Aim of this work is to prove that the Riesz transform operator is bounded on $L^p$ spaces, for every $p \in (1,\infty)$, on the class of $\RCD(K,\infty)$ metric measure spaces $(\sfX, \dist, \mm)$, with $K \in \mathbb{R}$. Writing $K_- = \max\cur{0, -K}$ and $\nor{\cdot}_p$ for the $L^p(\sfX, \mm)$ norm, and further referring to \cref{sec:notation} below for the notation of $\RCD$ spaces and related objects, our main result can be stated as follows:

\begin{theorem}\label{thm:main}
Let $(\sfX,\dist,\mm)$ be an $\RCD(K,\infty)$ space, for some $K \in \mathbb{R}$. Then, for every $p \in (1, \infty)$ and $f \in W^{1,2}(\sfX)$, it holds
\[ 
\nor{ \d f }_p \le c_p \nor{ \sqrt{ K_- - \Delta } f }_p,
\]
and
\[ \nor{ \sqrt{ K_- - \Delta } f }_p \le \sqrt{K_-} \nor{f}_p + c_p \nor{ \d f }_p,\]
where $c_p = 16 \max\cur{p, p/(p-1)}$.
\end{theorem}

As a straightforward application of our main result in combination with the arguments in \cite{ambrosio_lusin-type_2018}, we obtain the following quantitative Lusin approximation of Sobolev functions by Lipschitz functions on $\RCD(K,\infty)$ spaces (without \cref{thm:main}, this was previously limited to the case $p=2$).

\begin{corollary}\label{cor:lusin}
Let $(\sfX,\dist,\mm)$ be an $\RCD(K,\infty)$ space, with $K \in \mathbb{R}$, and let  $p \in (1, \infty)$.  Then, there exists $c = c(p, K)< \infty$ such that the following holds. For every  $f\in W^{1,p}(\sfX)$, there exists $g \in L^p(\sfX)$ with
\[ 
\nor{g}_p \le c \bra{ \nor{ f}_{p} + \nor{ \d f}_{p} } 
\]
and an $\mm$-negligible set $N\subset \sfX$ such that, for every $x,y\in \sfX \setminus N$ with $\dist(x,y)\leq 1/ K_-^2$, one has
\begin{equation}\label{eq:lusin-approx}
|f(x)-f(y)|\leq  \dist(x,y) (g(x)+g(y)).
\end{equation}
\end{corollary}

\subsection{Related literature}

The literature studying bounds on Riesz transform operators under geometric conditions is extensive, hence we focus our discussion on the works most relevant to the setting considered in this paper. In the seminal work \cite{bakry_transformation_1985}, D.\ Bakry investigated the $L^p$-boundedness of Riesz transform operators for general symmetric Markovian semigroups, not necessarily of diffusion type, satisfying an abstract curvature condition formulated via the so-called $\Gamma_2$ operator. He subsequently specialized his results to smooth Riemannian manifolds \cite{bakry_etude_1987}, where a lower bound on $\Gamma_2$ is seen to be  equivalent to a uniform lower bound of the  Ricci curvature. Such curvature lower bounds do not depend on the dimension of the manifold and in fact Bakry's work was connected to P.\ A.\ Meyer's original proof of Riesz inequalities on Wiener space \cite{meyer_transformations_1984}, using a probabilistic approach yielding representation formulae in terms of stochastic processes and bounds on square functions of associated martingales. Several extensions and improvements of the method have been proposed in the literature, e.g.\ to higher order Riesz transforms on manifolds or Wiener space \cite{yoshida_sobolev_1992, shigekawa_sobolev_1992}. The explicit value of the $L^p$-norm of the Riesz transform operator in the Euclidean setting was first computed in \cite{iwaniec_riesz_1995}: this raised the question of improved estimates, in particular with respect to the dependence upon $p \in (1,\infty)$, in the setting of weighted manifolds \cite{li_martingale_2008, banuelos_martingale_2013}. Eventually, this line of research lead to asymptotically sharp estimates, in the work  \cite{carbonaro_bellman_2013} by the first author and O.\ Dragi{\v c}evi{\'c}, with a completely analytic proof, in the philosophy of \cite{nazarov_hunt_1996, dragicevic_linear_2012}.

\vspace{1em}

The class of $\RCD(K,\infty)$ spaces was first introduced in \cite{ambrosio_metric_2014} to rule out Finsler geometries that are non-Riemannian from the class of $\CD(K, \infty)$ metric measures spaces. The latter class, introduced in \cite{Lott-Villani09, Sturm06I, Sturm06II}, consists of spaces where the entropy functional is $K$-convex along Wasserstein $W_2$ geodesics: this variational condition on the space reduces, on weighted Riemanniann manifolds, to a lower bound $\Gamma_2 \ge K$. Looking back at \cite{bakry_transformation_1985}, it is not a completely unexpected result that a suitably defined Riesz transform on $L^p$ spaces must be bounded when the underlying geometry is $\RCD(K,\infty)$. 

However, in such a generality, such problem seemed to be open until very recently. Indeed, Riesz transform bounds on the smaller class of $\RCD(K,N)$ spaces, i.e.\ having a finite dimension $N \in [1,\infty)$, have been investigated in \cite{jiang_heat_2016}, relying on heat kernel bounds or Poincaré and doubling properties, but these are not available in the $N=\infty$ case.  
A first step towards solving the general $\RCD(K, \infty)$ case can be ascribed to the preprint \cite{li2019littlewood}, where  boundedness of the ``vertical'' square functions for the heat and Poisson semigroups are established on $\RCD(K,\infty)$ spaces, a key tool in Bakry's approach to the problem. In the more recent   work \cite{esaki2023littlewood}, a similar boundedness result is established in the much wider class of ``tamed'' spaces \cite{erbar2022tamed}, namely Dirichlet spaces with a \emph{distributional} synthetic Ricci curvature lower bound. Finally, in the preprint \cite{esaki2023riesz}, which appeared on the arXiv just few days before our work, such square function bounds have been applied to a well-known integral representation for the Riesz transform (see \cref{lem:integral-rep}), eventually leading to the first proof of Riesz transform bounds on $\RCD(K, \infty)$ spaces. This approach is close in spirit to that of Bakry and, in particular, it should be affected by the same issues concerning non-sharp asymptotic dependence with respect to $p$ of the bounds, in the limits $p\to 1^+$ and $p \to \infty$.

\subsection{Comments on our proof technique}

Our approach  does not make use of square functions instead and is completely independent of \cite{li2019littlewood,esaki2023littlewood,esaki2023riesz}. Indeed, it is closer to that of the first author and O.\ Dragicevic in the already mentioned \cite{carbonaro_bellman_2013}, i.e.\ by making use of suitable Bellman functions. This technique originates in stochastic analysis in the works of D.L.\ Burkholder \cite{burkholder_sharp_nodate} and was popularized in harmonic analysis by F.L.\ Nazarov, S.R.\ Treil and A.\ Volberg (see \cite{nazarov_hunt_1996}). In particular, it leads  to explicit constants and asymptotically sharp dependence with respect to the exponent $p$ of the operator norm bounds.

Compared with \cite{carbonaro_bellman_2013}, one is tempted to directly employ the same Bellman functions, with the appropriate modifications. However, in the $\RCD$ setting, the need of regularization (e.g.\ to rigorously apply the chain rule and perform integration by parts)  has to be addressed in a different way, see \cref{sec:bell} for details. When compared with \cite{carbonaro_bellman_2013}, another difference arises in order to deal with the case of $\mm(\sfX)=\infty$, for which  we introduce suitable weights to reduce ourselves to the case of finite measure.




%
%

\subsection{Further questions}

Although \cref{thm:main} may be established using different approaches, i.e.\ by using Littlewood-Paley theory as in  \cite{esaki2023riesz} or perhaps by developing enough stochastic calculus tools to reproduce Bakry's original proof, to our knowledge our work introduces for the first time Bellman functions techniques in the $\RCD$ setting, and may pave the way for further investigations. In particular, we list the following points.
\begin{enumerate}
\item Our proof already yields sharp asymptotic dependence in the bounds as $p\to 1^+$ and $p \to \infty$, but of course we do not conjecture our bounds to be sharp for fixed $p$. We expect however that by replacing our Bellman function with a more carefully crafted one, the argument may  lead to even better constants. In fact, it is well known that the use of Bellman functions \cite{burkholder_sharp_nodate, oskekowski2012sharp} leads to optimal constants for several inequalities for martingales and singular integrals.
\item Our work shows that the scope and flexibility of Bellman function techniques apply well beyond the case of Riemannian manifolds. Hence, it could be worth investigating further settings, even when boundedness of Riesz transforms have been addressed by other tools, such as discrete and non-commutative ones \cite{junge2010noncommutative, banuelos_martingale_2013, junge2018noncommutative}.
\item It is known that the boundedness of Riesz transform has relevant applications in elliptic or parabolic PDEs and for the $L^p$-Hodge decomposition theory: applications in this directions have not yet been investigated in the $\RCD$ setting.
\item In \cite[Theorem 5.2]{ambrosio_lusin-type_2018}, it is given an application of the Lusin approximation theorem to Sobolev vector fields in Gaussian Hilbert spaces, in order to establish quantitative bounds for the associated regular Lagrangian flows. Both our \cref{thm:main} and \cref{cor:lusin} are stated for real-valued functions, but it is a well-known principle in harmonic analysis that inequalities for singular integrals hold also for maps
with values in Hilbert spaces \cite[§II.5, Theorem 5]{stein_topics_1970} where the singular integral operator is applied componentwise. We thus conjecture that, with natural modifications in the definitions, but essentially the same proofs, our results extend also to Hilbert-valued maps, and following the proof of \cite[Theorem 5.2]{ambrosio_lusin-type_2018} one may establish quantitative bounds for regular Lagrangian flows in Hilbert spaces, endowed with any log-concave probability measure.
\end{enumerate}

\subsection{Structure of the paper} The exposition is organized as follows. In \cref{sec:notation} we recall the most relevant concepts and results concerning calculus on $\RCD$ spaces (\cref{sec:rcd}), Poisson semigroups (\cref{sec:poisson}) and Bellman functions (\cref{sec:bell}). \cref{sec:main} is devoted to the proof of \cref{thm:main}, while \cref{sec:lusin} discusses how \cref{cor:lusin} follows from \cref{thm:main} from the arguments of \cite{ambrosio_lusin-type_2018}. We end the paper with \cref{sec:examples} where we provide  examples of applications of our result to explicit differential operators, e.g.\  of Bessel (\cref{sec:bessel}) and Laguerre (\cref{sec:laguerre}) type, with a throughout derivation of the fact that they fit in the $\RCD$ setting. To give an explicit infinite dimensional class of examples, we finally specialize \cref{thm:main} in the case of Hilbert spaces endowed with a log-concave probability measure.

\section{Notation and basic facts}\label{sec:notation}

\subsection{Self-adjoint operators and spectral calculus} Let $(H,\langle\cdot,\cdot\rangle)$ be a (separable) Hilbert space. Given a (possibly unbounded) self-adjoint operator $A: D(A) \subseteq H \to H$, i.e.\ densely defined, closed and symmetric, the spectral theorem  (see e.g.\ \cite[Section 10.1]{hall2013quantum}) provides a projection-valued measure $(E_\lambda)_{\lambda \in \mathbb{R}}$ such that
\[ 
A v = \int_{-\infty}^\infty \lambda \de E_\lambda(v) \qquad \text{for every $v\in D(A)$.}
\]
This in turn allows to define Borel functional calculus, yielding e.g.\ the definition of operators $f(A)$ for $f: \mathbb{R} \to \mathbb{R}$ Borel, with domain
\[ 
D(f(A)) = \cur{ v \in H \, : \, \int_{-\infty}^\infty |f(\lambda)|^2 \de \langle E_{\lambda} (g), g \rangle < \infty}.
\]

We also write
\[ 
\sfR(A) = \cur{ Av\, : \, v \in D(A) }, \qquad \operatorname{Ker}(A) = \cur{v \in D(A): Av = 0}
\]
respectively for the image and the kernel of $A$, and recall that
\[ \overline{\sfR(A)} = \operatorname{Ker}(A)^{\perp}, \]
i.e., the closure in $H$ of the image is the orthogonal (in $H$) to the kernel.

\subsection{$\RCD$ spaces}\label{sec:rcd} 
We refer to \cite{gigli_nonsmooth_2018} for a detailed introduction to calculus on metric measure spaces and specifically for those satisfying a synthetic Ricci curvature lower bound and enjoying a Riemannian-like behaviour, i.e.\ $\RCD$ spaces. Here we only recall basic definitions and properties that will be used below as well as more specific bibliographical references.

\subsubsection*{Sobolev spaces}
We let throughout $(\sfX,\dist)$ be a complete and separable length metric space endowed with a non-negative Radon measure $\mm$ with $\supp(\mm) = \sfX$ and which is finite on bounded sets. By $\Lip(\sfX)$ we shall denote the space of Lipschitz functions $f : \sfX \to \R$ with bounded support and by $\Meas(\sfX)$ the set of signed Radon measures with finite total variation on bounded sets of $\sfX$,  endowed with the total variation norm $\nor{\cdot}_{\TV}$. For the definition of the {\bf Sobolev class} $S^2(\sfX)$ (for brevity, we omit to specify $\dist$ and $\mm$ in what follows) and of {\bf minimal weak upper gradient} $|D f|$ we adopt the point of view from \cite{ambrosio2013density} which shows equivalence between several approaches, such  as \cite{cheeger1999differentiability} and \cite{shanmugalingam2000newtonian}. The Sobolev space $W^{1,2}(\sfX)$ is then defined as $L^2(\sfX) \cap S^2(\sfX)$ and when endowed with the norm $\|f\|_{W^{1,2}}^2:=\|f\|_{L^2}^2+\||Df|\|_{L^2}^2$, $W^{1,2}(\sfX)$ is a Banach space. We also define $W^{1,p}(\sfX)$, $p \in (1,\infty)$, as the closure of $W^{1,2}(\sfX)$ with respect to the norm $\|f\|_{W^{1,p}}^p:=\|f\|_{L^p}^p+\||Df|\|_{L^p}^p$. The {\bf Cheeger energy} is the convex and lower semicontinuous functional $E:L^2(\sfX)\to[0,\infty]$ given by
\[
E(f):=\left\{\begin{array}{ll}
\displaystyle{\frac12\int_{\sfX}|D f|^2\,\de\mm}&\qquad \text{for }f\in W^{1,2}(\sfX), \\
+\infty&\qquad\text{otherwise}.
\end{array}\right.
\]
Starting from these definitions, one can define a concept of Sobolev capacity $\Cap$ as in the Euclidean setting. In particular, by density of $\Lip(\sfX)$ in $W^{1,2}(\sfX)$ (see \cite{ambrosio2013density}), every Sobolev function $f \in W^{1,2}(\sfX)$ admits a $\Cap$-precise representative $\tilde{f}$ (in the $\mm$-a.e.\ equivalence class) obtained by approximating $f$ with a sequence of Lipschitz functions converging in $W^{1,2}(\sfX)$ (possibly extracting a subsequence).

\subsubsection*{Infinitesimally Hilbertian spaces}
From now on we will always assume that  $(\sfX,\dist,\mm)$ is {\bf infinitesimally Hilbertian}, i.e., that $W^{1,2}(\sfX)$ is Hilbert.  In this case $E$ is a Dirichlet form and we denote by $\Delta$ its infinitesimal generator, which is a closed self-adjoint operator on $L^2(\sfX)$, with domain $D(\Delta)\subset W^{1,2}(\sfX)$. The {\bf heat flow} in this setting is the linear Markovian semigroup $\h_t = \exp(t\Delta)$ associated to $E$, which is also equivalent to the $W_2$-gradient flow of the relative entropy ${\rm Ent}_\mm$: for any $f\in L^2(\sfX)$ the curve $t \mapsto \h_tf \in L^2(\sfX)$ is continuous on $[0,\infty)$, locally absolutely continuous on $(0,\infty)$ and the only solution of
\[
\ddt \h_tf = \Delta\h_tf \qquad \h_tf \to f\text{ as }t \downarrow 0.
\]

\subsubsection*{Cotangent and Tangent modules}
Using the theory of $L^\infty$-modules, in \cite{gigli_nonsmooth_2018} it is suitably defined the space of square integrable sections of the cotangent bundle $T^*\sfX$, called {\bf cotangent module} $L^2(T^*\sfX)$ (and more generally $L^p(T^* \sfX)$, $p \in [1, \infty]$). The {\bf differential} is then a linear map $\d$ from $S^2(\sfX)$ with values in $L^2(T^*\sfX)$, $f \mapsto \d f$, such that $\mm$-a.e.\ $|\d f| = |Df|$, the former being the pointwise norm in the module. Due to the infinitesimal Hilbertianity assumption, the cotangent module is a Hilbert space, i.e.\ the pointwise norm is $\mm$-a.e.\ induced by a scalar product $\ang{\cdot, \cdot}$,  and is canonically isomorphic to its dual, the {\bf tangent module} $L^2(T\sfX)$, both in the sense of Hilbert spaces and of modules. We use the musical notation for such isomorphism $\sharp: L^2(T^*\sfX) \to L^2(T \sfX)$ (its inverse $\flat$ will not be used). For example, in \cite{gigli_nonsmooth_2018} the gradient $\nabla f$ is defined as $(\d f)^\sharp$, but we avoid this notation to work only with differentials and covariant derivatives, as defined below. Finally, let us point out that, thanks to the structure theorem \cite[Theorem 1.4.11]{gigli_nonsmooth_2018}, we can always identify elements in the tangent or cotangent module as Borel maps with values in a Hilbert space (possibly infinite dimensional): this allows us to generalize many measure-theoretical results (such as duality and interpolation theorems) to this setting. As an example, we recall the duality:
\begin{equation}\label{eq:duality-general-forms} 
\| \omega \|_{L^p(T^*\sfX)}  = \sup_{\substack{\eta \in L^2 \cap L^q(T^*\sfX) \\ \|\eta\|_q \leq 1}}\int_\sfX\langle \omega,\eta\rangle\,\de\mm \qquad \text{for conjugate exponents $\frac 1 p + \frac 1 q= 1$.}
\end{equation}
As regards the properties and the calculus rules satisfied by the differential operators defined so far, let us recall the {\bf locality} of the differential, i.e.\ $\d f = \d g$ $\mm$-a.e.\ on $\{f=g\}$ for all $f,g \in S^2(\sfX)$, as well as the chain and Leibniz rules
\begin{align*}
\d(\varphi\circ f) & = (\varphi'\circ f) \d f && \forall f \in S^2(\sfX),\ \varphi:\R\to \R\ \text{Lipschitz}\\
\d(fg) &= g\,\d f + f\,\d g && \forall f,g \in L^\infty \cap S^2(\sfX)
\end{align*}
where it is part of the properties the fact that $\varphi\circ f,fg\in S^2(\sfX)$ for $\varphi,f,g$ as above. Moreover, $\test\sfX$, $D(\Delta)$ and Lipschitz functions with bounded support are dense classes of functions in $W^{1,2}(\sfX)$. We shall use these facts extensively without further mention and address the reader to \cite{ambrosio_metric_2014} and \cite{gigli_nonsmooth_2018} for their proof.

We occasionally need to consider the larger class $H^{1,1}(\sfX)$ of functions $f \in L^1(\sfX)$ that can be approximated in $L^1$ by a sequence $(f_n)_n \subseteq W^{1,2}\cap L^1(\sfX)$ with $\d f_n \to g \in L^1(T^*\sfX)$. We let then $g = \d f$, since the limit is unique and coincides with the differential if $f \in W^{1,2}(\sfX)$. In particular we have, for $f \in W^{1,2}(\sfX)$, $|f|^2 \in H^{1,1}(\sfX)$ and $\d |f|^2 = 2 f \d f$.

\subsubsection*{Entropy and Ricci bounds}

The curvature-dimension condition $\CD(K,\infty)$, independently proposed by Sturm \cite{Sturm06I, Sturm06II} and by Lott-Villani \cite{Lott-Villani09}, comes then into play as $K$-displacement convexity of the relative entropy functional ${\rm Ent}_\mm$ along $W_2$-geodesics, namely if for every $\mu_0,\mu_1 \in D({\rm Ent}_\mm)$ there exists a $W_2$-geodesic $(\mu_t)_{t \in [0,1]} \subset \mathscr{P}_2(\sfX)$ connecting them which satisfies
\[
{\rm Ent}_\mm(\mu_t) \leq (1-t){\rm Ent}_\mm(\mu_0) + t{\rm Ent}_\mm(\mu_1) - \frac{K}{2}t(1-t)W_2^2(\mu_0,\mu_1), \qquad \forall t \in [0,1],
\]
as after \cite{SturmVonRenesse05} in a smooth Riemannian manifold this is equivalent to ${\rm Ric} \geq K$. Recall that ${\rm Ent}_\mm : \mathscr{P}_2(\sfX) \to [-\infty,+\infty]$ is defined as
\[
{\rm Ent}_\mm(\mu) := \left\{\begin{array}{ll}
\displaystyle{\lim_{\eps \downarrow 0}\int_{\rho > \eps}\rho\log\rho\,\de\mm} & \qquad \text{if } \mu = \rho\mm \\
+\infty&\qquad\text{otherwise}
\end{array}\right.
\]
and $D({\rm Ent}_\mm) := \{\mu \in \mathscr{P}_2(\sfX) \,:\, {\rm Ent}_\mm(\mu) \in [-\infty,+\infty)\}$.

The triple $(\sfX,\dist,\mm)$ is an $\RCD(K,\infty)$ space if it is infinitesimally Hilbertian and satisfies the $\CD(K,\infty)$ condition.

\begin{remark}[$\h_t$ is conservative]\label{rem:conservative}
If $(\sfX, \dist, \mm)$ is an $\RCD(K, \infty)$ space, then the heat flow is conservative, i.e.\ it maps the set of probability measures into itself or, by duality, the identity $\h_t 1 = 1$ holds. By \cite[Theorem 1.6.6]{fukushima2010dirichlet}, this is also equivalent to the existence of a sequence $(g_n)_{n \in \N} \subseteq L^1 \cap W^{1,2}(\sfX)$ such that $0\le g_n \le 1$, $g_n \uparrow 1$ and
\[
\lim_{n \to \infty}  | \d  g_n | = 0 \quad \text{weakly in $L^2(\sfX)$.}
\]
The construction is rather simple: let $g_n$ solve $(1-\Delta)g_n = f_n$ where $f_n \in L^1\cap L^2(\sfX)$, $0 \le f_n \le 1$ and let the sequence $(f_n)_{n \in \N}$ pointwise grow to 1. 
\end{remark}

\subsubsection*{Test forms, covariant derivative and exterior differential}

In an $\RCD(K,\infty)$ space, a finer differential structure is at disposal. Indeed one can single out sufficiently rich classes of test objects, namely the vector spaces of {\bf test functions}  and {\bf test $1$-forms}, respectively defined in \cite{gigli_nonsmooth_2018} as
\[
\begin{split}
& \test\sfX := \Big\{ f \in D(\Delta) \cap L^{\infty}(\sfX) \ :\ |\d f| \in L^{\infty}(\sfX),\ \Delta f \in W^{1,2}(\sfX) \Big\}, \\
& \testf\sfX := \Big\{ \sum_{i=1}^n g_i\d f_i \,:\, n \in \N,\, f_i,g_i \in \test\sfX \, \forall i=1,...,n \Big\},
\end{split}
\]
Then, via a distributional approach, it is possible to provide the notion of {\bf covariant derivative} $\nabla$ on $\testf\sfX$,
\[ 
\nabla: \testf\sfX \subseteq L^2( T^* \sfX) \to L^2( (T^*)^{\otimes 2} \sfX), \quad  \nabla \omega (\cdot , X) = \nabla_X \omega,
\]
and the {\bf exterior differential} $\d$ and its dual {\bf codifferential} $\d^*$,
\[ 
\d: \testf\sfX \to L^2(\Lambda^2 T^* \sfX ), \quad  \quad \d^* :  \testf\sfX \to L^2(\mm)
\]
with the associated Sobolev spaces $W^{1,2}_{C}(T^*\sfX)$,  $W^{1,2}_{\H}(T^*\sfX)$, and norms
\[
\begin{split} 
\|\omega\|^2_{W^{1,2}_{C}} & := \|\omega\|^2_{L^2(T^*\sfX)} + \|\nabla \omega\|^2_{L^2((T^*)^{\otimes 2}\sfX)}, \\
\|\omega\|^2_{W^{1,2}_{\H}} & := \|\omega\|^2_{L^2(T^*\sfX)} + \|\d\omega\|^2_{L^2(\Lambda^2 T^*\sfX)} + \|\d^*\omega\|^2_{L^2}
\end{split}
\]
(see \cite{gigli_nonsmooth_2018} for precise definitions, also of the spaces $L^2(T^{\otimes 2} \sfX)$, $L^2(\Lambda^2 T^*\sfX)$ and their ``Hilbert-Schmidt'' norms). It is an open problem to determine whether $\testf\sfX$ is dense in such Sobolev spaces: for this reason, in what follows we denote by $H^{1,2}_{C}(T^*\sfX)$, $H^{1,2}_{\H}(T^*\sfX)$ the respective closures. Usual calculus rules hold in these spaces (by density of test forms): we will use e.g.\ the chain rule. For $\omega, \omega' \in H^{1,2}_{C}(T^*\sfX)$, one has $\ang{\omega, \omega'} \in H^{1,1}(\sfX)$, with
\begin{equation}\label{eq:chain-rule-forms}  
\d \ang{\omega, \omega'}  = \nabla \omega ((\omega')^\sharp, \cdot) + \nabla \omega'( \omega^\sharp, \cdot),
\end{equation}

\subsubsection*{Heat flow on forms}
The augmented Hodge energy functional on 1-forms $E_{\H} : L^2(T^*\sfX) \to [0,\infty]$, defined as
\[
E_{\H}(\omega) := \begin{cases}
\displaystyle{\frac12\int_{\sfX}\Big(|\d\omega|^2 + |\d^*\omega|^2\Big)\de\mm} & \qquad \text{for }\omega \in H^{1,2}_H(T^*\sfX),\\
+\infty&\qquad\text{otherwise,}
\end{cases}
\] 
is a closed quadratic form on $L^2(T^* \sfX)$. Its sub-differential, called Hodge Laplacian $\Delta_\sfH$, is then represented as $\Delta_\sfH = \d\d^* + \d^*\d$ on $\testf{\sfX}$,  and the associated linear symmetric semigroup on $L^2(T^*\sfX)$ is called {\bf heat flow on $1$-forms}  $(\h_{{\H},t})$. For any $\omega \in L^2(T^*\sfX)$ the curve $t \mapsto \h_{{\H},t}\omega \in L^2(T^*\sfX)$ is continuous on $[0,\infty)$, locally absolutely continuous on $(0,\infty)$ and the only solution of
\[
\ddt \h_{{\H},t}\omega = -\Delta_{\H}\h_{{\H},t}\omega \qquad \h_{{\H},t}\omega \to \omega\text{ as }t \downarrow 0.
\]
For any $f \in W^{1,2}(\sfX,\sfd,\mm)$ and $\omega \in H^{1,2}_{\H}(T^*\sfX)$ it holds, for every $t \ge 0$,
\begin{equation}\label{eq:commutation-heat}
\h_{{\H},t}(\d f) = \d\h_t f, \quad \text{and by duality} \quad  \d^*\h_{{\H},t}\omega = \h_t(\d^*\omega).
\end{equation}

\subsubsection*{Bakry-Em\'ery inequalities} 

An equivalent statement of the curvature assumption is provided by the domination inequality between $(\h_{\H, t})_{t \geq 0}$ and the Markov semigroup $(\h_t)_{t \geq 0}$:
\begin{equation}\label{eq:domination-pointwise}
|\h_{{\H},t}\omega|^2 \leq e^{-2  Kt } \h_t|\omega|^2,\quad \mm{\rm -a.e.} \qquad \forall t \geq 0
\end{equation}
holds for all $\omega \in L^2(T^*\sfX)$  \cite[Proposition 3.6.10]{gigli_nonsmooth_2018}. By Jensen inequality, we obtain that for $p \in [2, \infty)$, $\omega \in L^2\cap L^p(T^* \sfX)$,
\[
\|\h_{{\H},t}\omega\|^p_{L^p(T^*\sfX)} \leq e^{-Kt} \|\h_t|\omega|^2\|^{p/2}_{L^{p/2}(\sfX)} \leq e^{-Kt}  \||\omega|^2\|^{p/2}_{L^{p/2}(\sfX)} = e^{-Kt} \|\omega\|^p_{L^p(T^*\sfX)}.
\]
For $p = \infty$ the same bound follows directly from \eqref{eq:domination-pointwise} since $\h_t |\omega|^2 \leq \|\omega\|^2_\infty$, while for $p \in [1,2)$ we argue by duality using \eqref{eq:duality-general-forms},
 \[
\begin{split}
\|\h_{{\H},t}\omega\|_{L^p(T^*\sfX)} & = \sup_{\substack{\eta \in L^2 \cap L^q(T^*\sfX) \\ \|\eta\|_q \leq 1}}\int_\sfX\langle\h_{{\H},t}\omega,\eta\rangle\,\de\mm = \sup_{\substack{\eta \in L^2 \cap L^q(T^*\sfX) \\ \|\eta\|_q \leq 1}}\int_\sfX\langle\omega,\h_{{\H},t}\eta\rangle\,\de\mm \\
& \leq \sup_{\substack{\eta \in L^2 \cap L^q(T^*\sfX) \\ \|\eta\|_q \leq 1}}\|\omega\|_{L^p(T^*\sfX)}\|\h_{{\H},t}\eta\|_{L^q(T^*\sfX)} \leq e^{-Kt} \|\omega\|_{L^p(T^*\sfX)}.
\end{split}
\]
Moreover, one can closely follow the arguments in \cite[Theorem III.1]{stein_topics_1970}, recalling e.g.\ that by the structure theorem  \cite[Theorem 1.4.11]{gigli_nonsmooth_2018} we can identify elements in $L^2(T^*\sfX)$ as Borel maps with values in a Hilbert space. Therefore, for any $p \in (1, \infty)$ and $\omega \in L^2 \cap L^p( T^* \sfX)$ the map $[0,\infty) \ni t \mapsto \h_{\H, t}$ takes values in $L^2 \cap L^p(T^* \sfX)$ and can be analytically extended for all complex $t \in \mathbb{C}$ such that
\[
|{\rm arg}(t)| \leq \frac{\pi}{2}\Big(1 - \Big|\frac{2}{p} - 1\Big|\Big).
\]
In particular, smoothness of $(0, \infty) \ni t \mapsto \h_{{\H},t} \omega$ can and will be always assumed.

\subsubsection*{Bochner formula}

A further relevant definition for our purposes is that of {\bf measure-valued Laplacian} \cite[Definition 3.1.2]{gigli_nonsmooth_2018}. Given a function $f\in W^{1,2}(\sfX)$ and a Radon signed measure $\mu \in \Meas(\sfX)$ with finite total variation on bounded sets, we write $f \in D(\bd)$ if
\begin{equation}\label{eq:ibp-def-meas-lap}
\int_\sfX g\,\de\mu = -\int_\sfX \langle\d g,\d f\rangle\,\de\mm,\qquad\quad \text{for every $g \in \Lip(\sfX)$.}
\end{equation}
Clearly, $\mu$ is unique if it exists,  and we denote $\mu = \bd f$. We notice that
\[
f \in D(\bd),\ \bd f \ll \mm \text{ and }\frac{\de\bd f}{\de\mm}\in L^2(\sfX) \qquad \Leftrightarrow \qquad f \in D(\Delta)\text{ and in this case }\Delta f = \frac{\de\bd f}{\de\mm}.
\]
Arguing as in \cite[Equation 3.6.20]{gigli_nonsmooth_2018} it is not difficult to see that $\bd f$, although possibly singular with respect to $\mm$, must be absolutely continuous with respect to the Sobolev capacity $\Cap$.

The relevance of this definition is due to the {\bf Bochner identity} on $\RCD(K,\infty)$ spaces, where it reads as follows: up to canonical isomorphisms between $L^2(T\sfX)$ and $L^2(T^*\sfX)$, for every $\omega \in \testf\sfX$, it holds $|\omega|^2/2 \in D(\bd)$ with 
\begin{equation}\label{eq:bochner} 
\bd \frac{|\omega|^2}{2} = \bra{ -  \ang{\omega, \Delta_{\H} \omega } + | \nabla \omega|^2} \mm + \Ric(\omega),
\end{equation}
where $\Ric: \sqa{H^{1,2}_H(T^*\sfX)}^2 \to \Meas(\sfX)$ is a (quadratic) continuous map such that, in the sense of measures,
\begin{equation}\label{eq:ric-bound} 
\Ric(\omega, \omega) \ge K |\omega|^2\mm \quad \forall \omega \in H^{1,2}_H(T^*\sfX).
\end{equation}
Notice that usually (also in \cite{gigli_nonsmooth_2018}), $\Ric$ is defined on the tangent module, but here for simplicity we omit to specify a $\sharp$ operation and work only on the cotangent one.



\subsubsection*{A weaker version of Bochner formula}

For later purpose, it is useful to recast Bochner identity \eqref{eq:bochner} in a weaker form, thus allowing for a wider class of admissible forms. In order to do so, let us first integrate \eqref{eq:bochner} multiplied by a function $g \in \Lip(\sfX)$ with $g \geq 0$
\begin{equation}\label{eq:bochner-integral} 
\int_\sfX \ang{\d\omega,\d(g\omega)} + \ang{\d^*\omega,\d^*(g \omega)} - \frac 1 2\ang{\d |\omega|^2,\d g}\de\mm = \int_\sfX g |\nabla\omega|^2\de\mm + \int_\sfX g\de\Ric(\omega).
\end{equation}
Secondly, if we combine \eqref{eq:bochner-integral} with \eqref{eq:ric-bound}, we get
\[
\int_\sfX \ang{\d\omega,\d(g\omega)} + \ang{\d^*\omega,\d^*(g \omega)} - \frac 1 2\ang{\d |\omega|^2,\d g}\de\mm \geq \int_\sfX g\big(|\nabla\omega|^2 + K|\omega|^2\big)\de\mm
\]
and if we evaluate this inequality with $g_n (x) := (1-\d(B_n(x_0), x))^+$, $n \geq 1$ (where $x_0 \in \sfX$ is an arbitrary fixed point), then by $0 \le g_n \uparrow 1$, $|\d g_n| \le \chi_{B_{n+1}(x_0) \setminus B_n(x_0)} \to 0$ and the dominated convergence theorem we obtain
\[ 
2E_{\H}(\omega) \geq \int_{\sfX} | \nabla \omega|^2 \de\mm + K \int_{\sfX} |\omega|^2 \de\mm
\]
for all $\omega \in \testf\sfX$. Leveraging on this and on the continuity of $\Ric: \sqa{H^{1,2}_H(T^*\sfX)}^2 \to \Meas(\sfX)$, it is possible to pass to the limit in \eqref{eq:bochner-integral} with respect to $W^{1,2}_H(T^*\sfX)$-convergence, thus obtaining the validity of \eqref{eq:bochner} for all $\omega \in H^{1,2}_H(T^*\sfX) \cap L^\infty(T^*\sfX)$.




\subsubsection*{A $\Cap$-precise chain rule}


%
Let us show the following chain rule for functions with measure-valued Laplacian (which are in particular absolutely continuous with respect to $\Cap$).

\begin{lemma}[chain rule]\label{lem:chain-rule}
Let $g = (g^i)_{i=1}^d \in D(\bd) \cap  L^\infty (\sfX)$ and $F \in C^2_\sfb(\R^d)$. Then $F \circ g \in D(\bd) \cap L^\infty(\sfX)$ with 
\[ 
\bd F \circ  g   = \sum_{i=1}^d \partial_i F\bra{ \tilde g }\bd g^i + \sum_{i,j=1}^d \partial_{ij}^2 F\bra{ g} \ang{ \d  g^i, \d  g^j} \mm,
\]
where $\tilde g = (\tilde g^i)_{i=1}^d$ denotes the $\Cap$-precise representative of $g$. 
\end{lemma}

\begin{proof}
For simplicity, we prove the thesis in the case $d=1$ only. We notice that $F'\bra{ \tilde g }$ is a $\Cap$-precise representative, since $F' \in C^1_\sfb(\R)$. Therefore, if $h \in\Lip(\sfX)$, we can integrate by parts
\[  
\begin{split}  
\int_{\sfX} h  &  \de \sqa{ F'\bra{ \tilde g }\bd g + F''\bra{g} |\d g|^2 \mm} \\
& \quad = - \int_{\sfX} \ang{\d ( h  F'\bra{\tilde g},\d g} \de\mm + \int_{\sfX} h F''(g) |\d g|^2 \de \mm \\
& \quad = -\int_{\sfX} F'(g) \ang{\d h, \d g} \de \mm = - \int_{\sfX}\ang{\d h, \d (F \circ g) } \de \mm,
\end{split}
\]
obtaining the validity of \eqref{eq:ibp-def-meas-lap}.
\end{proof}

From this lemma and \eqref{eq:bochner} we have that, for $\omega \in H^{1,2}_{\H}(T^* \sfX) \cap L^\infty(T^* \sfX)$ and $F \in C^2_\sfb(\R)$,
\begin{equation}\label{eq:bochner-g} \bd F\bra{ \frac{|\omega|^2} {2} } =  F'\left(\frac{|\omega|^2} {2}\right) \sqa{ - \ang{\omega,  \Delta_{\H} \omega} + |\nabla \omega|^2 + \Ric(\omega)} +  F''\bra{\frac{|\omega|^2} {2}}\abs{ \d \frac{ |\omega|^2 }{2} }^2  ,\end{equation}
(where we choose the precise representative for $\abs{\omega}^2/2$).

%

\subsection{Poisson semigroups}\label{sec:poisson}  

In this section we borrow some notation from \cite{carbonaro_bellman_2013} for (shifted) Poisson semigroups on functions and forms, defined for $t \geq 0$ on an $\RCD(K, \infty)$ metric measure space as
\[ 
\cL^a := - \Delta + a \Id, \quad \sfP_t^a := \exp(-t \sqrt{\cL^a}), \quad \text{and} \quad  \vecL^a :=  \Delta_{\H}+ a \Id, \quad \vec \sfP_t := \exp( -t (\vec \cL^a)^{1/2} ),
\]
where $a \ge 0$. To lighten the notation, we will often omit the identity operator ${\rm Id}$. These semigroups can be rigorously defined respectively on $L^2(\sfX)$ and $L^2(T^* \sfX)$ via spectral calculus:
\begin{equation}\label{eq:spectral-subordination}
\sfP_t ^a f = \int_{0}^\infty \exp( -t (\lambda+a)^{1/2}) \de E_\lambda f, \quad  \quad \vec \sfP_t^a  \omega = \int_{0}^\infty \exp( -t (\lambda+a)^{1/2}) \de \vec E_\lambda \omega
\end{equation}
for $f \in L^2(\sfX)$, $\omega \in L^2(T\sfX)$, where $(E_\lambda )_{\lambda}$, $(\vec E_{\lambda})_\lambda$ denote the resolution of the identity associated to the self-adjoint non-negative operators $-\Delta$ and $\Delta_{\H}$. We have the identities 
\[ 
\partial_t^2 \sfP_t^a f = \cL^a \sfP_t^a f \quad \text{and} \quad  \partial_t^2 \vec \sfP_t^a\eta = \vec \cL^a \vec \sfP_t^a \eta, \quad \text{for $t \in (0, \infty)$,}
\]
respectively for $f \in L^2(\sfX)$, $\eta \in L^2(T^*\sfX)$, and moreover $\sfP_t^a f \in D(\Delta) \subseteq W^{1,2}(\sfX)$ and $\vec \sfP_t\eta^a  \in D( \Delta_{\H}) \subseteq H^{1,2}_{\H}(T^* \sfX)$. 

\subsubsection*{Subordination formulas} Alternatively, we can use Bochner subordination from the corresponding heat semigroups, 
\begin{equation}\label{eq:bochner-subordination} 
\sfP_t^a f  = \int_0^\infty  \h_{t^2/4s} f \frac{e^{-s-at^2/4s}}{\sqrt{\pi s}} \de s, \qquad \vec \sfP^a_t \omega = \int_0^\infty  \h_{H, t^2/4s} \omega \frac{e^{-s -at^2/4s}}{\sqrt{\pi s}} \de s.
\end{equation}
The representation above is useful to deduce the contraction and regularity properties from those of the heat semigroups (see also \cite[Lemma 2]{carbonaro_bellman_2013}). To keep the notation simple, let us introduce the function
\[ 
\rho_a(t) = \int_0^\infty \frac{e^{-s-at^2/4s}}{\sqrt{\pi s}} \de s, 
\]
and notice that $\rho_a(t) \le \rho_0(t) = 1$ for every $t$, and, if $a>0$,
\begin{equation} \label{eq:integration-rho-at} 
\int_0^\infty \rho_a(t) \de t =  \frac{1}{\sqrt{a}} \int_0^\infty e^{-s} \int_{-\infty}^\infty \frac{e^{-at^2/4s}}{\sqrt{2\pi 2s/a}} \de t \de s  = \frac{1}{\sqrt{a}}.
\end{equation}

\begin{lemma}\label{lem:poisson}	
Let $(\sfX,\dist,\mm)$ be $\RCD(K,\infty)$. For every $a \ge 0$, the following properties hold:
\begin{enumerate}[i)]
\item for every $f \in W^{1,2}(\sfX)$, $\omega \in H^{1,2}_{\H}(T^* \sfX)$ and $t \geq 0$,
\begin{equation}\label{eq:d-pt-vec} 
\d \sfP_t^a f  =  \vec \sfP_t^a \d f \quad \text{ and }  \quad \d^* \vec \sfP_t^a \omega =  \sfP_t^a \d^* \omega,
\end{equation}
\item for every $p \in [1, \infty]$, $f \in L^2 \cap L^p(\sfX)$, $\omega \in L^2 \cap L^p(T^* \sfX)$ and $t \geq 0$,
\[ 
\|\sfP_t^a  f\|_p \leq  \rho_a(t)\|f\|_p \quad \text{and} \quad \|\vec \sfP_t^{(a+K)_+} \omega\|_p \leq  \rho_a(t)\|\omega\|_p,
\]
\item for every $p \in (1, \infty)$, $f \in L^2 \cap L^p(\sfX)$ and $\omega \in L^2 \cap L^p(T^* \sfX)$ the functions $t \mapsto \sfP_t f \in L^p(\sfX)$ and $t \mapsto \vec \sfP_t \omega \in L^p(T^*\sfX)$ are continuous on $[0,\infty)$ and smooth on $(0,\infty)$; if $p=\infty$, then they are weak-* continuous on $[0, \infty)$.
\end{enumerate}
\end{lemma}

\begin{proof}
The first claim follows from \eqref{eq:bochner-subordination} and \eqref{eq:commutation-heat}. For \emph{ii)} recall that $\|\h_t f\|_p \leq \|f\|_p$ for all $p \in [1,\infty]$ and notice that, by Minkowski integral inequality,
\[ 
\nor{ \sfP_t^a f}_p \le \int_0^\infty \nor{\h_{t^2/4s} f}_p \frac{e^{-s-at^2/4s}}{\sqrt{\pi s}} \de s  = \rho_a(t) \nor{f}_p.
\]
The second part of the statement follows by a similar argument, using $\nor{\h_{\H, t} \omega}_p \le e^{-K t }\nor{\omega}_p$ (previously obtained as a consequence of \eqref{eq:domination-pointwise}). 
Finally, as concerns \emph{iii)}, we first note that the claimed regularity properties hold true for $t \mapsto \h_t f$ and $t \mapsto \h_{\H,t}\omega$ by the already cited \cite[Theorem III.1]{stein_topics_1970}. To show that the semigroups $\sfP$ and $\vec \sfP$ inherit the same properties, it is sufficient to argue by dominated convergence theorem in the integral representations \eqref{eq:bochner-subordination}, so to commute differentiation and integration. More explicitly, we observe that
\[
\begin{split}
\frac{1}{h}\|\sfP_{t+h}^a f - \sfP_t^a f\|_p & \leq \int_0^\infty \frac{1}{h}\|\h_{(t+h)^2/4s}f - \h_{t^2/4s}f\|_p \frac{e^{-s-a(t+h)^2/4s}}{\sqrt{\pi s}}\de s \\
& \quad + \int_0^\infty \|\h_{t^2/4s}f\|_p\frac{e^{-s-a(t+h)^2/4s} - e^{-s-at^2/4s}}{h \sqrt{\pi s}}\de s
\end{split}
\]
and remark that in the first integral, $\frac{e^{-s-a(t+h)^2/4s}}{\sqrt{\pi s}}$ is bounded by an integrable function, uniformly in $|h| < t/2$, while $t \mapsto \h_{t^2/4s}f$ is $C^\infty$ as $L^p$-valued map, so that $h^{-1}\|\h_{(t+h)^2/4s}f - \h_{t^2/4s}f\|_p$ converges locally uniformly as $h \to 0$, hence it is locally bounded, uniformly in $|h| < t/2$: this allows to apply the dominated convergence theorem and pass to the limit as $h \to 0$ inside the integral. As concerns the second integral, we use again the fact that $\|\h_{t^2/4s} f\|_p \leq \|f\|_p$ and note that
\[
\frac{e^{-s-a(t+h)^2/4s} - e^{-s-at^2/4s}}{h \sqrt{\pi s}} = \frac{e^{-s-at^2/4s}}{\sqrt{\pi s}}\frac{e^{-ath/2s-ah^2/4s} - 1}{h} \leq C\frac{e^{-s-at^2/4s}}{\sqrt{\pi}s^{3/2}},
\]
with $C$ depending continuously on $a,t$ but independent of $s$, so that also in this case the dominated convergence theorem can be applied. The argument for $\vec \sfP$ is completely analogous.
\end{proof}


\subsubsection*{Vertical and horizontal derivatives} We introduce the following notation for derivatives of functions $(f_t)_{t \ge 0}$, $f_t = \sfP_t^a f$, $f \in L^2(\sfX)$  and forms $(\eta_t)_{t \ge 0}$,  $\eta_t = \vec \sfP_t^a \eta$, $\eta \in L^2\cap L^\infty(\sfX)$:
\[ \begin{split} \hat f &:= \frac{f}{\abs{f}} \chi_{\cur{|f| \neq 0}}, \qquad \partial_t^{\parallel} f := \hat f \partial_t f, \qquad \nabla^\parallel f := \hat f \d f, \\
 \hat \eta & := \frac{\eta}{\abs{\eta}} \chi_{\cur{\abs{\eta} \neq 0}}, \qquad \partial_t^{\parallel} \eta  := \ang{ \hat{\eta}, \partial_t \eta},\qquad   \nabla^{\parallel} \eta := \nabla \eta ( \hat{\eta}^\sharp, \cdot ).\end{split}\]
We also write 
\[
\begin{split} 
| \partial_t^{\perp} f|^2 & :=   |\partial_t f |^2 - | \partial_t^\parallel f |^2 \qquad |\nabla^{\perp} f |^2 :=  |\d  f |^2 - |\nabla^\parallel f|^2 \\
|\partial_t^{\perp} \eta|^2 & :=   |\partial_t \eta |^2 - | \partial_t^\parallel \eta|^2 \qquad |\nabla^{\perp} \eta|^2 :=  |\nabla  \eta |^2 - |\nabla^\parallel \eta|^2 \end{split}
\]
and finally
\[ 
|\overline{\nabla} f|^2 :=  |\partial_t f|^2 + |\d f|^2 =  |\partial_t f|^2 + |\nabla^{\parallel} f|^2 + |\nabla^{\perp} f|^2,
\]
and similarly for $\eta$.

Using this notation, we have e.g.\
\begin{equation}\label{eq:first-order-identities} \partial_t \frac{|\eta|^2}{2} = \ang{ \eta, \partial_t \eta}  = \abs{\eta} \partial_t^{\parallel} \eta, \qquad \d \frac{|\eta|^2}{2}  =  \nabla \eta (\eta^\sharp, \cdot) =  \abs{\eta}\nabla^{\parallel} \eta, \end{equation}

%


\subsubsection*{A chain rule inequality} Combining \cref{lem:chain-rule}, Bochner formula \eqref{eq:bochner} and inequality \eqref{eq:ric-bound}, we obtain the following result.

\begin{lemma}[chain rule inequality]\label{lem:convexity}
Let $(\sfX,\dist,\mm)$ be an $\RCD(K,\infty)$ space and let $a \ge -K$. For any $d \ge 1$, $\eta = (\eta^i)_{i=1}^d \in \bra{ L^2\cap L^\infty(\sfX)} \cup \bra{ L^2\cap L^\infty(T^*\sfX)}$, write $\eta_t = (\eta^i_t)_{i=1}^d$ for the (componentwise) evolution with the Poisson semigroup $\sfP^a$ (in case of functions) $\vec \sfP^a$ (in case of forms) and $|\eta_t| = (|\eta^i_t|)_{i=1}^d$. Let
\[ 
F \in C^2( [0, \infty)^d; \R)  \text{ be convex with $\partial_i F \ge 0$, for $i \in \cur{1, \ldots, d}$,}
\]
and $w \in W^{1,2}(\sfX)$, $w \ge 0$. Then, for every $t >0$, $F \circ |\eta_t|  \in W^{1,2}(\sfX)$ and the following inequality holds in the sense of distributions on $t \in (0, \infty)$,
\begin{equation}\label{eq:convexity} \begin{split} 
\partial_t^2 \int_{\sfX} \bra{  F \circ |\eta_t|}   w \de \mm \ge  &  \int_{\sfX} \sum_{i,j=1}^d \partial_{ij}^2 F \bra{ \partial_t^\parallel \eta^i_t  \partial_t^\parallel \eta^j_t + \ang{ \nabla^\parallel \eta^i_t, \nabla ^\parallel\eta^j_t} } w \de \mm \\
& + \int_{\sfX} \sum_{i=1}^d \partial_i F |\eta_t^i|^{-1}\bra{ | \partial_t ^{\perp}\eta^i_t |^2 +  | \nabla ^{\perp}\eta^i_t |^2} w \de \mm \\
& +  \int_{\sfX} \ang{\d  (F\circ |\eta_t|), \d  w }  \de \mm,
\end{split}
\end{equation}
where, in the right-hand side, all the partial derivatives of $F$ are evaluated at $|\eta_t|$.
\end{lemma}


\begin{proof} 
By density, it is sufficient to assume that $w \in \Lip(\sfX)$. Moreover, to simplify notation, we assume that $\eta^i \in L^2 \cap L^\infty(T^*\sfX)$ (we point out below the minor differences in the general case).

For $\eps>0$, we introduce the following regularization of $z \mapsto \sqrt{ 2 z}$, $z \ge 0$,
\[ 
\sigma_\veps(  z ) := (\veps+ 4z^2)^{1/4} - \veps^{1/4}.
\]
Notice that $0 \leq \sigma_\eps (z)\le \sqrt{ 2 z}$ and $\lim_{\eps\to 0} \sigma_{\eps}(z)= \sqrt{2z}$. Moreover, since
\[ 
\sigma'_\eps(z) = \frac{ 2 z}{(\eps + 4z^2)^{3/4}} \qquad \text{and} \qquad \sigma''_\eps(z) = \frac{-4z^2 + 2\eps}{(\eps+ 4z^2)^{7/4}},
\]
we have
\begin{equation}\label{eq:g'} 
0 \le  z \sigma'_\veps(z^2/2)  = \frac{ z^3}{(\veps+ z^4)^{3/4}} \le 1, \qquad  \lim_{\veps \to 0} z \sigma'_\veps(z^2/2) = \chi_{\cur{z \neq 0}},
\end{equation}
and
\begin{equation}\label{eq:g''} 
z^2 \sigma''_{\veps}( z^2/2) +  \sigma'_\veps(z^2/2)  = \frac{ 3\varepsilon}{z(\veps + z^4)^{7/4}} \ge 0.
\end{equation}
To keep notation simple, let us also write
\[ 
g_\veps = (g^i_\veps)_{i=1}^d := \sigma_{\veps} (|\eta^i|^2/2), \quad g'_\veps := (\sigma'_\veps(|\eta^i|^2/2))_{i=1}^d, \quad g''_\veps := (\sigma''_\veps(|\eta^i|^2/2))_{i=1}^d.
\]
Let us prove first that $F \circ |\eta| \in W^{1,2}(\sfX)$. To this end, it is sufficient to apply the chain rule to $F \circ g_\veps$, i.e.\
\begin{equation}\label{eq:derivative-convexity} 
\d (F \circ g_\veps) = \sum_{i=1}^d \partial_i F \d g^i_\veps = \sum_{i=1}^d \partial_i F  (g'_\veps)^i \abs{\eta^i} \nabla^\parallel \eta^i_t,
\end{equation}
(we omit to specify that the partial derivatives of $F$ are evaluated at $g_\veps$). For $t>0$, we have $|\d \eta^i_t| \in L^2(\sfX)$ and $|\eta_t^i| \in L^\infty(\sfX)$, hence by \eqref{eq:g'}  it follows that \eqref{eq:derivative-convexity} converges in $L^2(TX)$, i.e.\ $F\circ |\eta| \in W^{1,2}(\sfX)$ (even with uniform bounds on compact subsets of $(0, \infty)$.).

To obtain \eqref{eq:convexity}, it is sufficient to prove, for every $\veps>0$, the following inequality
\begin{equation}\label{eq:convexity-approx} 
\begin{split} 
\partial_t^2 \int_{\sfX} F \bra{ (g^i_\veps)_{i=1}^d} w \de \mm \ge  &  \int_{\sfX} \sum_{j,k=1}^d \partial_{ij}^2 F \, (g'_\veps)^j |\eta^j_t| \, (g'_\veps)^k |\eta^k_t|\cdot  \\
& \quad \quad  \cdot \bra{ \partial_t^\parallel \eta^j_t  \partial_t^\parallel \eta^k_t + \ang{ \nabla^\parallel \eta^j_t, \nabla^\parallel \eta^k_t} } w \de \mm \\
& + \int_{\sfX} \sum_{j=1}^d \partial_j F (g'_\veps)^j |\eta^j_t| |\eta_t^j|^{-1}\bra{ | \partial_t^\perp \eta^j_t|^2 +  | \nabla^\perp \eta^j_t|^2} w \de \mm \\
& + \int_\sfX \ang{\d F \bra{ (g^i_\veps)_{i=1}^d}, \d w} \de \mm.
\end{split}
\end{equation}
Indeed, \eqref{eq:convexity} is then recovered by passing to the limit as $\eps\to 0$, since $(g'_\veps)^{i} |\eta^i| \to \chi_{\cur{\abs{\eta^i} \neq 0}}$ and Fatou's lemma ensures convergence in the sense of distributions, since the right-hand side is larger than $- \nor{\d (F \circ |\eta_t|)}_2 \nor{\d w}_2$ which is integrable (recall that $F \circ |\eta_t| \in W^{1,2}(\sfX)$ uniformly on compact subsets of $(0, \infty)$).

In what follows we argue for fixed $\veps>0$, hence for simplicity we drop it in the notation. We first collect some identities valid for every $i \in \cur{1, \ldots, d}$, that follows easily from the chain rule. Again, for simplicity, we also drop the superscript $i$.
\begin{equation} \label{eq:first-order} \partial_t g = g' \partial_t \frac{|\eta|^2}{2} = g' \abs{\eta} \partial_t ^\parallel \eta 
\end{equation}
and
\begin{equation}\label{eq:second-order}
\begin{split} 
\partial_t^2 g & = g'' \abs{ \partial_t \frac{|\eta|^2}{2}}^2 + g'  \partial_t^2 \frac{|\eta|^2}{2} \\
& = g'' |\eta|^2 |\partial_t^\parallel \eta |^2  + g' \sqa{\abs{ \partial_t \eta}^2 + \ang{\eta, \vec \cL^{a} \eta}}.
\end{split}
\end{equation}
Moreover, by \eqref{eq:bochner-g} we have the identity  between measures (functions should be regarded as densities with respect to $\mm$)
\begin{equation}
g' \ang{\eta, \vec \cL^a \eta} = \cL^{0} g  +  g''|\eta|^2 |\nabla^ \parallel\eta|^2 + g' \sqa{ |\nabla \eta|^2 + \Ric(\eta) +a |\eta|^2},\end{equation}
(if $\eta$ is a function, the term $\Ric(\eta)$ is missing). Substituting this identity into \eqref{eq:second-order}, we conclude that
\begin{equation}\label{eq:key-identity}
\begin{split} 
\partial_t^2  g  & = g'' | \eta |^2 \sqa{ |\partial_t^\parallel \eta|^2  + | \nabla^\parallel \eta|^2 } +g'\sqa{ |\partial_t \eta|^2 + |\nabla \eta|^2 + \Ric(\eta)+a |\eta|^2 } + \cL g\\
& = g'\sqa{ |\partial_t^\perp \eta|^2 + |\nabla^\perp \eta|^2} + \bra{g'' | \eta |^2 + g'} \sqa{ | \partial_t^\parallel \eta|^2 + |\nabla ^\parallel\eta |^2} \\
& \quad + g' \sqa{\Ric(\eta)+a|\eta|^2} + \cL g
\end{split}
\end{equation}
and we observe that all but the last addends in the right-hand side above are non-negative.

We have all the tools to prove \eqref{eq:convexity}. Indeed, by the chain rule (recall that $t \mapsto \eta_t$ are smooth),
\[ \partial_t^2 \int_{\sfX} (F \circ g)  w \de \mm = \int_{\sfX} \sum_{i,j=1}^d \partial_{ij}^2 F \partial_t g^i \partial_t g^j  w \de \mm + \int_\sfX \sum_{i=1}^d \partial_i F \partial_t^2 g^i w \de \mm.\]
(For simplicity, we avoid to specify that the derivatives of $F$ are evaluated at $g$.) By \eqref{eq:first-order}, we already have that
\[  
\int_{\sfX} \sum_{i,j=1}^d \partial_{ij}^2 F\,  \partial_t g^i \partial_t g^j w \de \mm =   \int_{\sfX} \sum_{i,j=1}^d \partial_{ij}^2 F \,  (g')^i |\eta^i_t| \, (g')^j |\eta^j_t| \, \partial_t ^\parallel\eta^j_t \partial_t ^\parallel\eta^k_t  w \de \mm,
\]
hence we focus on the remaining terms 
\[ 
\int_{\sfX} \sum_{i=1}^d \partial_i F \partial_t^2  g^i w \de \mm.
\]
For any $i \in \cur{1, \ldots, d}$, we integrate \eqref{eq:key-identity} multiplied by the precise representative of $(\partial_i F\circ g)w \ge 0$. Since
\[ 
\int_{\sfX} w \partial_i F    \sqa{ \bra{g'' | \eta |^2 + g'} \bra{ | \partial_t^\parallel \eta|^2  + | \nabla ^\parallel\eta |^2}\de \mm + g' \bra{ \de \Ric(\eta^\sharp) +a |\eta|^2 \de \mm  }}\ge 0, 
\]
we obtain that
\[  
\int_{\sfX} \partial_i F  \partial_t^2  g^i  w \de \mm \ge  \int_{\sfX}\partial_i F   (g')^i \sqa{ |\partial_t^\perp \eta^i|^2  + |\nabla^\perp \eta^i|^2} w \de \mm  + \int_{\sfX} w \partial_i F \cL g^i.
\]
Summing upon $i \in \cur{1,\ldots, d}$, we recover the last two lines in \eqref{eq:convexity-approx}. Indeed, integrating by parts which is allowed by the discussion above),
\[
\begin{split} 
\int_{\sfX} w \sum_{i=1}^d \partial_i  F \circ g \cL g^i & = \int_{\sfX} \sum_{i,j=1}^d \partial_{ij}^2   F\ang{ \d g^i, \d g^j} w \de \mm +\int_{\sfX} \sum_{i=1}^d \partial_{i}   F \ang{\d w, \d g^i } \de \mm \\
& = \int_{\sfX} \sum_{i,j=1}^d \partial_{ij}^2 F \, (g')^i |\eta^i_t| \,  (g')^j |\eta^j_t| \,  \ang{ \nabla^\parallel \eta^i, \nabla^\parallel \eta^j} w \de \mm \\
& \quad \quad + \int_{\sfX} \ang{  \d  (F\circ g) , \d  w} \de \mm.
\end{split}
\]
\end{proof}

We end this section with the following simple result that can be also obtained using the structure theorem~\cite[Theorem 1.4.11]{gigli_nonsmooth_2018}.

\begin{lemma}\label{lem:positive-form}
Let $d \ge 1$, $A, B: \sfX \to \R^{d \times d}$ be Borel, with values $\mm$-a.e.\  in the set of symmetric matrices. If $A \ge B$ $\mm$-a.e.\ (in the sense of quadratic forms), then for any $\eta (\eta^i)_{i=1}^d \in L^2(T^* X)$ one has
\[ \sum_{i,j=1}^d A_{ij}(x) \ang{ \eta^i(x), \eta^j(x)} \ge \sum_{i,j=1}^d B_{ij}(x) \ang{ \eta^i(x), \eta^j(x)} \quad \text{$\mm$-a.e.\ $x \in \sfX$.}\]
\end{lemma}
\begin{proof}
It is sufficient to prove the case $B=0$. We can decompose $A(x) = C(x)^\tau C(x)$, with $x \mapsto C(x)$ Borel, so that $\mm$-a.e.\ on $\sfX$,
\[ \sum_{i,j=1}^d A_{ij} \ang{ \eta_i, \eta_j}  = \abs{ \sum_{i=1}^d C_{ij} \eta_i }^2 \ge 0.\qedhere \]
\end{proof}

\subsection{Bellman function}\label{sec:bell}

In this section we recall some properties of the following Bellman function $\bell(u,v)$, introduced in \cite{dragivcevic2011bilinear}, defined for $u$, $v\ge0$ as
\begin{equation}
\label{eq:bellman}
\bell\bra{u,v} := u^p + v^q + \delta \begin{cases} \displaystyle{\frac 2 p u^p + \bra{ \frac 2 q -1} v^q} & \quad \text{if $u^p \ge v^q$} \\
u^2 v^{2-q} & \quad \text{if $u^p\le v^q$,}
\end{cases}
\end{equation}
where $p \ge 2$, $q = p/(p-1)$ and $\delta := q(q-1)/8$. Young inequality with exponents $p/2$ and $p/(p-2) = q/(2-q)$ yield, for every $u$, $v \geq 0$,
\begin{equation}\label{eq:bound-bad-good-bellman}
u^2 v^{2-q} \le \frac 2 p u^p + \bra{ \frac 2 q -1} v^q,
\end{equation}
whence the upper bound
\begin{equation}
\label{eq:bellman-upper-bound}
\bell\bra{u,v} \le \bra{1 + \frac 2 p} u^p + \frac 2 q v^q.
\end{equation}
Moreover $\bell$ is $C^1$ for $u$, $v>0$, with
\[ 
\begin{split} 
\partial_u \bell(u,v) & = p u^{p-1} + \delta  \begin{cases} 2 u^{p-1} &\quad \text{if $u^p \ge v^q$} \\
2u v^{2-q} & \quad \text{if $u^p \le v^q$,}
\end{cases}\\
\partial_v \bell(u,v) & = q v^{q-1} + \delta  \begin{cases} (2-q) v^{q-1} &\quad \text{if $u^p \ge v^q$} \\
(2-q)u^2 v^{1-q} & \quad \text{if $u^p \le v^q$.}
\end{cases}
\end{split}
\]
In particular we have the inequalities
\begin{equation}\label{eq:bellman-derivatives-upper-bound} 
\partial_u \bell(u,v) \le p(1+\delta) u^{p-1} + 2 \delta v, \qquad \partial_v\bell(u,v) \le (q+\delta (2-q)) v^{q-1}.
\end{equation}
In the next lemma we recall some convexity and monotonicity properties of the function $\bell$.

\begin{lemma}\label{lem:bellman-convexity}
Let $\delta = q(q-1)/8$. Then the following inequalities hold for every $u>0$, $v>0$ with $u^p \neq v^q$:
\begin{equation}\label{eq:first-derivative-bellman} u^{-1} \partial_u \bell\bra{u,v} \ge 2\delta v^{2-q}, \quad  v^{-1} \partial_v \bell\bra{u,v} \ge 2\delta v^{q-2},\end{equation}
\begin{equation}\label{eq:second-derivative-bellman}
\partial_{uu}^2 \bell\bra{u,v} \alpha^2 + 2 \partial_{uv}^2\bell\bra{u,v} \alpha \beta + \partial_{vv}^2 \bell\bra{u,v} \beta^2 \ge  \delta \bra{ v^{2-q} \alpha^2 + v^{q-2} \beta^2 },
\end{equation}
for every $\alpha, \beta \in \R$.
\end{lemma}

Inequality \eqref{eq:second-derivative-bellman} shows that $\bell$ is convex in the region $\cur{u^p \le v^q}$. Using \eqref{eq:bound-bad-good-bellman} and the fact that $\bell$ is $C^1$, it follows easily that $\bell$ is globally convex. Moreover, one could also prove that $\bell \in W^{2,1}_{loc}((0, \infty)^2)$, i.e., its distributional derivative actually coincides with the a.e.\ defined derivative (but we do not need this fact).

\begin{proof}
Taking into account the explicit formulas for $\partial_u\bell$ and $\partial_v\bell$ written above, it is easy to see that, on the one hand, in $\{u^p > v^q\}$ it holds
\[
u^{-1}\partial_u\bell\bra{u,v} = (p+2\delta)u^{p-2} \geq (p+2\delta)v^{q(1-2/p)} = (p+2\delta)v^{2-q} \geq 2\delta v^{2-q}
\]
and in $\{u^p < v^q\}$
\[
u^{-1}\partial_u\bell\bra{u,v} = pu^{p-2} + 2\delta v^{2-q} \geq 2\delta v^{2-q},
\]
whence the first lower bound in \eqref{eq:first-derivative-bellman}. On the other hand, in $\{u^p > v^q\}$ we have
\[
v^{-1}\partial_v\bell\bra{u,v} = (2\delta + q(1-\delta))v^{q-2} \geq 2\delta v^{q-2}
\]
since $p \geq 2$ and thus $\delta \in (0,1)$, while in $\{u^p < v^q\}$
\[
v^{-1}\partial_v\bell\bra{u,v} = qv^{q-2} + \delta(2-q)u^2 v^{-q} \geq qv^{q-2},
\]
as $q \leq 2$, and since $q \geq 8\delta$ by the very definition of $\delta$, $v^{-1}\partial_v\bell\bra{u,v} \geq 2\delta v^{q-2}$ follows and thus also the second lower bound in \eqref{eq:first-derivative-bellman}.

For the second part, see \cite[p. 60]{nazarov_hunt_1996}.
\end{proof}

For technical reasons, we introduce the following approximate versions of $\bell$, for $\veps \in (0,1]$, using a smooth non-negative function $\varrho_\varepsilon$ supported on $(\veps, 2\veps)$ with unit integral:
\[ 
\bell_\varepsilon\bra{u,v} := \int_{(\veps, 2\veps)} \bell(u+u',v+v') \varrho_\varepsilon(u')\varrho_{\varepsilon}(v') \de u' \de v',
\]
so that $\bell_{\varepsilon}$ is smooth on $[0,\infty)^2$. In particular, we gain regularity even for $v=0$, at the price of the following  analogue of inequality \eqref{eq:bellman-upper-bound} for $\bell_{\varepsilon}$: 
\begin{equation}\label{eq:bellman-upper-bound-worse} 
\bell_{\varepsilon}(u,v)  \le \bra{1 + \frac 2 p} \sqa{u}_\veps^p   + \frac 2 q \sqa{v}_\veps^q 
\end{equation}
having defined, for any $r \in \R$, $u \ge 0$,
\begin{equation}\label{eq:notation-ugly-eps} 
\sqa{u}_\veps^r := \int_{\veps}^{2 \veps} (u+u')^r \varrho_\veps(u') \de u'.
\end{equation}
Similarly, we have the analogue of \eqref{eq:bellman-derivatives-upper-bound},
\begin{equation}
\label{eq:bellman-derivatives-upper-bound-approx}
\partial_u \bell_\veps(u,v) \le p(1+\delta) \sqa{u}^{p-1}_\veps + 2 \delta \sqa{v}_\veps, \quad \partial_v\bell_\veps(u,v) \le (q+\delta (2-q)) \sqa{v}^{q-1}_\veps
\end{equation}   
Moreover, it is not difficult to directly check that the following analogues of \eqref{eq:first-derivative-bellman} and \eqref{eq:second-derivative-bellman} hold, for every $u$, $v\ge0$,
\begin{equation}\label{eq:first-derivative-bellman-approx} u^{-1} \partial_u \bell_\veps\bra{u,v} \ge \delta \sqa{v}_\veps^{2-q}, \qquad  v^{-1} \partial_v \bell_\veps\bra{u,v} \ge \delta \sqa{v}_\veps^{q-2},
\end{equation}
\begin{equation}
\label{eq:second-derivative-bellman-approx}
\partial_{uu}^2 \bell_{\varepsilon} \bra{u,v} \alpha^2 + 2 \partial_{uv}^2\bell_{\varepsilon}\bra{u,v} \alpha \beta + \partial_{vv}^2 \bell_{\varepsilon}\bra{u,v} \ge  \delta  \bra{ \sqa{v}_\veps^{2-q} \alpha^2 + \sqa{v}_\veps^{q-2} \beta^2 } 
\end{equation}
for every $\alpha, \beta \in \R$ (again for every $u$, $v\ge0$). Finally, we notice that the inequality $x^2 + y^2 \ge 2 |x y|$ entails the lower bound
\begin{equation}\label{eq:inequality-convexity-multiplicative} 
\begin{split} 
\sqa{v}_\veps^{2-q} \alpha^2 + \sqa{v}_\veps^{q-2} \beta^2 & = \int_{\veps}^{2 \veps} \bra{(v+v')^{2-q} \alpha^2 + (v+v')^{q-2} \beta^2}  \varrho_{\veps}(v') \de v' \\
& \ge 2 \int_{\veps}^{2 \veps} \abs{\alpha}\abs{\beta} \varrho_{\veps}(v') \de v' = 2 \abs{\alpha} \abs{\beta}.
\end{split}
\end{equation}

\section{Proof of \cref{thm:main}} \label{sec:main}

In this section we prove our main result, \cref{thm:main}. This can be obtained by the boundedness on $L^p$ spaces, $p \in (1, \infty)$ of the Riesz transform 
\[ 
\cR^a = \d (a + \cL)^{-1/2}
\]
on any $\RCD(K,\infty)$, for any $a\ge K_-$. This operator, often also called \emph{local}  Riesz transform if $a>0$, is initially defined on the image of the Sobolev space $W^{1,2}(\sfX)$ with respect to the map $(a+\cL)^{1/2}$, but it extends at once to a bounded operator with respect to convergence in $L^2(\sfX)$, with domain given by the  $L^2(\sfX)$-closure of the original domain. 
\begin{remark}\label{rem:domain}
Let us notice that this domain also  also coincides with $\operatorname{Ker}(a+\cL)^{\perp}$, because of the spectral decomposition
\[
L^2(\sfX) = \overline{ \sfR(  (a+\cL)^{1/2} )} \oplus \operatorname{Ker}(a+\cL).
\]
Furthermore, if $a>0$, then $\operatorname{Ker}(a+\cL)^{\perp} = L^2(\sfX)$, while if $a=0$ the space $\operatorname{Ker}(\cL)^{\perp}$ is either trivial (if $\mm$ is infinite) or contains only the constant functions (if $\mm$ is finite), since any element $f \in D(\cL)$ such that $\cL f = 0$ must have in particular $|\d f| = 0$, $\mm$-a.e.\ hence by locality of the differential $f$ is $\mm$-a.e.\ equal to a constant.
\end{remark}

To prove that $\cR^a$ is $L^2(\sfX)$-bounded, we begin by writing $f = (a+\cL)^{1/2} g$ and assuming for simplicity that $g \in D(\cL)$, an integration by parts and application of the spectral theorem then give
\[ 
\begin{split}
\int_{\sfX} | \cR^a f|^2 \de \mm  & = \int_{\sfX} | \d g |^2 \de \mm  =  \int_{\sfX} g \cL g = \int_0^\infty \lambda \de \ang{E_\lambda g, E_\lambda g} \\
& \le  \int_0^\infty (a+\lambda) \de \ang{E_\lambda g, E_\lambda g} =  \int_{\sfX} | (a +\cL)^{1/2} g |^2 \de \mm = \int_{\sfX } |f|^2 \de \mm.
\end{split}
\]
The restriction on $g$ can be removed by density, hence the $L^2(\sfX)$-boundedness follows.

As in \cite{carbonaro_bellman_2013}, the strategy of proof relies on Poisson semigroup interpolations, duality, and a bilinear embedding exploiting the growth and convexity properties of the Bellman function. The precise results are summarized in the following two propositions, from which \cref{thm:main} follows at once.

\begin{proposition}[integral representation]\label{lem:integral-rep}
The following identity holds, for every $f \in \sfR(a+\cL)$ and $\omega \in L^2(T\sfX)$:
\[ \int_{\sfX} \ang{ \cR^a  f, \omega} \de \mm  = 4 \int_0^\infty \int_\sfX \ang{ \d  \sfP_t^{a}  f,  \partial_t  \vec \sfP_t^{a} \omega} \de \mm \,  t \de t. \]
\end{proposition}

\begin{proposition}[bilinear embedding]\label{lem:bilinear-emb}
Let $p,q \in (1, \infty)$ be conjugate exponents. Then $f \in L^p(\sfX)$, $\omega \in L^q(T \sfX)$,
\[ \int_{0}^\infty \int_\sfX  | \overline{\nabla} \sfP_t^{a} f | |\overline{\nabla} \vec \sfP_t^{a} \omega|  \de \mm \,  t \de t\le 4\max\cur{p, q} \nor{f}_p \nor{\omega}_q.\] 
\end{proposition}

Before we prove the two propositions, let us briefly recall how \cref{thm:main} follows from them.

\begin{proof}[Proof of \cref{thm:main}]
Combining \cref{lem:integral-rep} and \cref{lem:bilinear-emb} yields
\[ \int_{\sfX} \ang{ \cR^a  f, \omega} \de \mm \le 16 \max\cur{p, q} \nor{f}_p \nor{\omega}_q.\]
If we consider the supremum over $\omega \in L^2\cap L^q(T^*\sfX)$ with $\nor{\omega}_q \le 1$, by duality we deduce that
\begin{equation}\label{eq:riesz-lp} \nor{ \cR^a f}_p \le 16 \max\cur{p,q} \nor{f}_p, \quad \text{for every $f \in \sfR(a+\cL)$,}\end{equation}
with $q = p/(p-1)$. By density and \cref{rem:domain} above, the inequality extends to $f \in \operatorname{Ker}(\cL)^{\perp}$.

Writing $f=(a +\cL)^{1/2}g$ for some $g \in W^{1,2}(\sfX)$, we find
\begin{equation}
\label{eq:riesz-one-side} \nor{ \d g}_p \le 16 \max \cur{p,q} \| (a +\cL)^{1/2} g \|_p, \quad \text{for every $g \in W^{1,2}(\sfX)$.}
\end{equation}
To obtain the converse inequality, we argue by duality. Consider first the case $a>0$. Then,
\[
\begin{split} 
\| (a +\cL)^{1/2} f\|_p & = \sup_{g \in L^2(\sfX) \,:\, \nor{g}_q \le 1} \int_{\sfX}  g  (a+\cL)^{1/2}f  \de \mm \\
& = \sup_{h \,:\, \nor{(a+\cL)^{1/2} h }_q \le 1} \int_{\sfX}  ((a +\cL)^{1/2} h) ( (a +\cL)^{1/2}f ) \de \mm 
\end{split}
\]
where the second equality follows because to any $g \in L^2 \cap L^q(\sfX)$ we can associate
\[ 
h = (a+\cL)^{-1/2} g = \int_0^\infty \sfP^a_t h \,\de t,
\]
which is well defined as an element of $L^2 \cap L^q(\sfX)$, with
\begin{equation}\label{eq:nor-h-nor-g} \nor{ h }_q \le \frac {1}{\sqrt{a}} \nor{ g }_q\end{equation}
as an application of \autoref{lem:poisson} and \eqref{eq:integration-rho-at}. Therefore, if $\nor{g}_q \le 1$, we obtain
\[ 
\begin{split}
\int_{\sfX}  ((a +\cL)^{1/2} h)  ( (a +\cL)^{1/2}f ) \de \mm  &  = \int_{\sfX} h( a f +\cL f ) \de \mm \\
& = \int_{\sfX} a fh + \ang{\d h, \d f}\de \mm \\
& \le a \nor{f}_p \nor{h}_q + \nor{\d f}_p \nor{\d h}_q \\
& \le \sqrt{a}\nor{f}_p + 16 \max\cur{p,q} \nor{\d f}_p,
\end{split}
\]
having used \eqref{eq:nor-h-nor-g}, and \eqref{eq:riesz-one-side} with $q$ instead of $p$ and $h$ instead of $f$. For the  case $K\ge 0$ we need t $a=0$, we simply let $a \to 0^+$ and notice that, for $f \in W^{1,2}(\sfX) $, by spectral calculus we have the  convergence in $L^2(\sfX)$:
\[ (a +\cL)^{1/2} f \to \cL^{1/2} f \quad \text{as $a \to 0^+$.}\]
Hence, by Fatou's lemma we obtain
\[ 
\begin{split}
\nor{\cL^{1/2} f}_p & \le \liminf_{a \to 0^+} \nor{(a +\cL)^{1/2} }_p  \le \liminf_{a \to 0^+} \sqrt{a}\nor{f}_p + 16 \max\cur{p,q} \nor{\d f}_p \\
& = 16 \max\cur{p,q} \nor{\d f}_p,
\end{split}
\]
provided that also $f \in L^p(\sfX)$. However, this condition can be later removed by a density argument hence also the second inequality in \cref{thm:main} is proved.
\end{proof}

The proof of \cref{lem:integral-rep} and  \cref{lem:bilinear-emb}  rely on the formal integration by parts identity
\begin{equation}\label{eq:ibp} \varphi(0) = \int_0^\infty \varphi''(t) \, t \de t.\end{equation}
More rigorously, we will let $\delta \to 0$, $M\to \infty$ in the identity, valid for functions $\varphi \in C^2(0, \infty)$,
\begin{equation}\label{eq:ibp-rigorous} \varphi(\delta) = \int_{\delta}^{M} \varphi''(t) \, t \de t  + \varphi(M) + \varphi'(\delta)\delta - \varphi'(M)M. \end{equation}
In the proof of \cref{lem:bilinear-emb} we relax  the identity \eqref{eq:ibp-rigorous} to the inequality
\begin{equation}\label{eq:ibp-rigorous-inequality} \varphi(\delta) \ge \int_{\delta}^{M} \psi(t) \, t \de t  + \varphi(M) + \varphi'(\delta)\delta - \varphi'(M)M, \end{equation}
that holds (by approximation) e.g.\ if $\varphi$ is locally Lipschitz, $\delta$, $M$ are differentiability points for $\varphi$, and $\psi \in C(0, \infty)$ is such that $\varphi'' \ge \psi$ in the sense of distributions on $(0, \infty)$. 

%

\begin{proof}[Proof of \cref{lem:integral-rep}]
To keep the notation simple, we provide the proof only in the $a=0$ case, the general case being analogous. Given $f = \cL g$ for some $g \in D(\cL)$, we consider the continuous function
\[ 
\varphi(t) := \int_{\sfX} \ang{ \vec \sfP_t \cR  f, \vec \sfP_t \omega} \de \mm, \quad t \in [0, \infty),
\]
where we write $\sfP_t = \sfP^0_t$, $\vec\sfP_t = \vec \sfP^0_t$ and $\cR=\cR^0$. 
Using \eqref{eq:d-pt-vec}, we have the commutation
\[  \vec \sfP_t \cR  f = \d \sfP_t \cL^{-1/2} f = \cR \sfP_tf\]
hence
\[\begin{split}  \varphi(t) & =  \int_{\sfX} \ang{ \vec \sfP_t \cR f, \vec \sfP_t \omega} \de \mm =  \int_{\sfX} \ang{ \sfP_t \cL^{-1/2} f,  \d ^* \vec \sfP_t \omega} \de \mm   = \int_{\sfX} \ang{ \sfP_t \cL^{-1/2} f,   \sfP_t \d^* \omega} \de \mm \\
& = \int_{\sfX} \ang{ \sfP_{2t} \cL^{-1/2} f,  \d^* \omega} \de \mm. \end{split}\]
Differentiating twice, we obtain the equivalent expressions
\[ \varphi'(t) =  2 \int_{\sfX} \ang{ \sfP_{2t}  f,  \d^* \omega} \de \mm = 2 \int_{\sfX} \ang{   \d f,  \vec \sfP_{2t} \omega} \de \mm\]
\[ \varphi''(t) = 4 \int_{\sfX} \ang{ \sfP_{2t}  \cL^{1/2} f,  \d^* \omega} = 4 \int_{\sfX} \ang{   \d  \sfP_t f,   \partial_t \vec \sfP_{t} \omega} \de \mm .\]
Therefore, it is sufficient to justify the limit as $\delta\to 0$, $M\to \infty$ in \eqref{eq:ibp-rigorous}. 

The assumption $f = \cL g$ together with the spectral theorem give
\[ 
\nor{  \sfP_{2t}  \cL^{-1/2} f}_2^2 = \nor{  \sfP_{2t}  \cL^{1/2} g}_2^2 = \int_0^\infty e^{-4t \lambda^{1/2}} \lambda \de \ang{E_{\lambda}g,E_{\lambda}g} \le t^{-2} \sup_{z \ge 0} e^{-4z}z^2 \nor{g}_2^2, 
\]
so that, by Cauchy-Schwarz inequality, 
\[ 
|\varphi(t)| \le  \nor{  \sfP_{2t}  \cL^{-1/2} f}_2 \nor{\d^* \omega}_2 \to 0 \quad \text{as $t \to \infty$.}
\]
Arguing similarly for the first derivative, we have that 
\begin{equation}\label{eq:decay-at-infinity} 
\nor{  \sfP_{2t}   f}_2^2 = \nor{  \sfP_{2t}  \cL g}_2^2 = \int_0^\infty e^{-4t \lambda^{1/2}} \lambda^2 \de \ang{E_{\lambda}g,E_{\lambda}g} \le t^{-4} \sup_{z \ge 0} e^{-4z}z^4 \nor{g}_2^2, 
\end{equation}
hence
\[ 
|\varphi'(t) t | \le 2 t  \nor{  \sfP_{2t}   f}_2 \nor{\d^* \omega}_2 \to 0 \quad \text{both as $t \to 0$ and as $t \to \infty$.}
\]
Finally, for the second derivative, we have that
\begin{equation}\label{eq:decay-at-infinity-derivative} 
\nor{  \sfP_{2t}  \cL^{1/2} f}_2^2 = \nor{  \sfP_{2t}  \cL^{3/2} g}_2^2 = \int_0^\infty e^{-4t \lambda^{1/2}} \lambda^3 \de \ang{E_{\lambda}g,E_{\lambda}g} \le t^{-6} \sup_{z \ge 0} e^{-4z}z^6 \nor{g}_2^2, 
\end{equation}
as well as inequality
\[ 
\nor{  \sfP_{2t}  \cL^{1/2} f}_2^2 = \int_0^\infty e^{-4t \lambda^{1/2}} \lambda \de \ang{E_{\lambda}f,E_{\lambda}f} \le t^{-2} \sup_{z \ge 0} e^{-4z}z^2 \nor{f}_2^2\]
so that, by Cauchy-Schwarz inequality, 
\[ |\varphi''(t)| \le 4\nor{ \sfP_{2t}  \cL^{1/2} f}_2 \nor{\d^* \omega}_2  = O( \min\cur{t^{-1}, t^{-3} }),
\]
hence $\int_0^\infty |\varphi''(t)| t \de t < \infty$. By dominated convergence, we can take the limit in \eqref{eq:ibp-rigorous} and deduce the thesis. 
\end{proof}

\begin{proof}[Proof of \cref{lem:bilinear-emb}]
As in \cite{carbonaro_bellman_2013},  the  main idea is to apply \eqref{eq:ibp} to
\[ 
\varphi(t) := \int_\sfX \bell ( | f_t|, |\omega_t|)  \de \mm ,
\]
if $p \ge 2$, otherwise we exchange the roles of $f$ and $\omega$, where $f_t = \sfP^a_t f$ and $\omega_t = \vec \sfP^a_t \omega$. Indeed, by \eqref{eq:bellman-upper-bound} we have
\[ 
\varphi(0) \le \frac{2}{p} \nor{f}_{p}^p + \bra{\frac 2 q -1} \nor{\omega}_q^q,
\]
while \cref{lem:convexity} with $F = \bell$ and $\eta^1 = f$, $\eta^2=\omega$ provides a lower bound that, by the monotonicity and convexity properties of $\bell$, yields
\begin{equation}\label{eq:convexity-bellman-integrated} 
\varphi''(t) \ge 2\delta \int_{\sfX}  | \overline{\nabla}  f_t | |\overline{\nabla}  \omega_t| \de \mm.
\end{equation}
By homogeneity, it is sufficient to rescale $f$ and $\omega$ in such a way that $\nor{f}_p = \nor{\omega}_q= 1$ to obtain the thesis.
 
However, the lack of smoothness of $\bell$ prevents a straightforward application of \cref{lem:convexity}, hence we replace it with the regularization $\bell_{\varepsilon}$ defined in \cref{sec:bell}, for $\varepsilon>0$,  and  let  then $\varepsilon\to 0$. But this step raises an integrability issue if the measure $\mm$ is not finite, since \eqref{eq:bellman-upper-bound} should then be replaced by \eqref{eq:bellman-upper-bound-worse}. To overcome this problem we weight the measure $\mm$: we consider a sequence $(w_n)_{n\ge 1} \in L^1\cap W^{1,2}(\sfX)$ as provided in \cref{rem:conservative},  write $\mm_n := w_n \mm$ and argue with $\mm_n$ instead of $\mm$. To conclude, we  let first $\varepsilon\to 0$ and then $n \to \infty$.

For $\veps \in (0,1]$, $n\ge 1$, we apply \eqref{eq:ibp-rigorous-inequality} to the  function
\[ 
\varphi_{\veps, n}(t) := \int_\sfX \bell_{\veps}( |f_t|, |\omega_t|)  \de \mm_n. 
\]
Indeed, by \eqref{eq:bellman-upper-bound-worse}, we have the $\mm$-a.e.\ inequality
\[
\begin{split} 
\bell_{\veps} & \le    \bra{1+\frac{2}{p}} \sqa{|f_t|}_\veps^p  + \frac 2 q \sqa{|\omega_t|}_\veps^q   \le    c \bra{ |f_t|^{p} + |\omega_t|^q + \veps }, 
\end{split} 
\]
where, here and below, $c  \ge 0$ denotes some constant depending on $p$, $q$ only, and $\bell_{\veps}$ together with its derivatives in always evaluated at $(|f_t|, |\omega_t|)$. Integrating with respect to the finite measure $\mm_n$  and using Cauchy-Schwarz inequality it follows that
\begin{equation}
\label{eq:bound-phi-eps-n} \begin{split} |\varphi_{\veps, n}(t)|   & \le  c \bra{ \nor{f_t}_p^{p}  +\nor{\omega_t}_q^{q}  + \veps \nor{w_n}_1}\\
& \le c \bra{ \nor{f}_\infty^{p-1} \nor{f_t}_2 \nor{w_n}_2 + \nor{\omega}_\infty^{q-1} \nor{\omega_t}_2 \nor{w_n}_2 + \veps \nor{w_n}_1}
\end{split}\end{equation}
(all the norms above being with respect to $\mm$). By \cref{lem:poisson}, it follows that $\varphi_{\veps, n}$ is well defined, and an application of dominated convergence yields also continuity on $[0, \infty)$. Arguing as in \eqref{eq:decay-at-infinity}  for $f = \cL^a g$ and analogously for $\omega = \vec \cL^a \eta$, we have that $\lim_{t \to \infty} \nor{f_t}_2 + \nor{\omega_t}_2 = 0$, hence the second inequality in \eqref{eq:bound-phi-eps-n} shows that
\begin{equation}\label{eq:limsup-phi-eps-n} 
\limsup_{t \to \infty} |\varphi_{\veps, n}(t)| \le c \veps \nor{w_n}_1.
\end{equation}
We then notice that $\varphi_{\veps, n}$ is locally Lipschitz in $(0, \infty)$, and arguing as we did above, we write first the $\mm$-a.e.\ inequalities
\begin{equation}\label{eq:inequalities-derivatives} 
|\partial_u \bell_\veps| \le  c  \bra{|f_t|^{p-1} + |\omega_t| + \veps } \qquad  |\partial_v \bell_\veps|  \le c  \bra{ | \omega_t|^{q-1} + \veps},
\end{equation}
which lead to the upper bound (at a.e.\ $t \in (0, \infty)$),
\[
\begin{split} 
|\partial_t \varphi_{\veps, n}(t)| & \le \int_{\sfX} \bra{|\partial_u \bell_\veps| |\partial_t f_t| + |\partial_v \bell_\veps| |\partial_t \omega_t|}  \de \mm_n  \\
&\le  c \bra{  \bra{\nor{f}_\infty^{p-1} + \nor{\omega}_\infty^{p-1} + \veps}\nor{\partial_t f_t}_2\nor{w_n}_2 +  \bra{\nor{\omega}_\infty^{q-1}+\veps} \nor{\partial_t \omega_t}_2 \nor{w_n}_2}
\end{split}.
\]
Arguing as in \eqref{eq:decay-at-infinity-derivative}, it follows that 
\begin{equation}\label{eq:decay-at-infinity-derivatives-proof-bellman}
\limsup_{t \to \infty} t^3 \bra{\nor{ \partial_t f_t}_2 + \nor{\partial_t \omega_t}_2} < \infty,
\end{equation}
hence $\lim_{t \to \infty} t |\varphi_{\veps, n}(t)| =0$.

The key point is then to prove the following analogue of \eqref{eq:convexity-bellman-integrated},  with $\varphi_{\veps, \gamma}$ instead of $\varphi$ and an extra term arising from the introduction of $w_n$:
\begin{equation} \label{eq:convexity-bellman-integrated-approx}
\varphi''_{\veps, \gamma}(t) \ge 2\delta \int_{\sfX}  | \overline{\nabla} f_t | |\overline{\nabla} \omega_t|   \de \mm_n + \int_{\sfX} \ang{\d \bell_{\veps}(|f_t|, |\omega_t|), \d w_n  } \de \mm.
\end{equation}
Let us notice that, by the chain rule and \eqref{eq:inequalities-derivatives},
\begin{equation}\label{eq:bound-nabla-b-eps} 
\nor{ \d \bell_{\veps}(|f_t|, |\omega_t|) }_2 \le c  \bra{  \bra{\nor{f}_\infty^{p-1} + \nor{\omega}_\infty^{p-1} + \veps}\nor{\d f_t}_2 +  \bra{\nor{\omega}_\infty^{q-1}+\veps} \nor{\d \omega_t}_2}. 
\end{equation}
Using that
\[ 
\nor{ \d f_t}_2 = \nor{ \cL^{1/2}f_t }_2 \le \nor{ (a+\cL)^{1/2}f_t }_2 = \nor{\partial_t f_t}_2,
\]
and similarly with $\omega$ instead of $f$, it follows from \eqref{eq:decay-at-infinity-derivatives-proof-bellman}  that the rightmost term in \eqref{eq:convexity-bellman-integrated-approx} is integrable with respect to the measure $t \de t$ on $(0, \infty)$. Therefore, once \eqref{eq:convexity-bellman-integrated-approx} is proved, we can let $\delta \to 0$ and $M \to \infty$ in \eqref{eq:ibp-rigorous-inequality} to deduce
\[ 
|\varphi_{\veps, n} (0) | \ge 2\delta \int_0^\infty \int_{\sfX}  | \overline{\nabla} f_t | |\overline{\nabla} \omega_t|   \de \mm_n t \de t - \int_0^\infty \abs{ \int_{\sfX} \ang{\d \bell_{\veps}(|f_t|, |\omega_t|), \d w_n  } \de \mm} t \de t - c \veps \nor{w_n}_1,
\]
having also used \eqref{eq:limsup-phi-eps-n}. As already mentioned above, we let first $\veps\to 0$ and then $n \to \infty$. Indeed, by the chain rule and  inequalities \eqref{eq:bellman-derivatives-upper-bound-approx} it is not difficult to deduce that, as $\veps\to 0$, \[ \bell_{\veps}(|f_t|, |\omega_t|) \to \bell( |f_t|, |\omega_t|) \quad \text{in $W^{1,2}(\sfX)$,}\]
and that an inequality akin to \eqref{eq:bound-nabla-b-eps} holds (without $\veps$). Moreover, as $n \to \infty$, we have for every $t \in (0, \infty)$,
\[ \int_{\sfX} \ang{\d \bell_{\veps}(|f_t|, |\omega_t|), \d w_n  } \de \mm \to 0,\]
hence by dominated convergence (we repetitively use \eqref{eq:decay-at-infinity-derivatives-proof-bellman}) we conclude that
\[ \lim_{n \to \infty} \lim_{\veps \to 0} \int_0^\infty \abs{ \int_{\sfX} \ang{\d \bell_{\veps}(|f_t|, |\omega_t|), \d w_n  } \de \mm} t \de t =0 .\]
On the other side, 
\[ 
\lim_{n \to \infty} \lim_{\veps \to 0} \varphi_{\veps, n}(0) =\varphi(0),
\]
hence the thesis follows.

We are only left with the proof of  \eqref{eq:bellman-derivatives-upper-bound-approx}. We apply  \cref{lem:convexity} with $F = \bell_{\veps}$, $\eta^1 = f$, $\eta^2=\omega$  and $w = w_n$, which yields
\begin{equation}\label{eq:inequality-temp-main-proof}\begin{split} \varphi''_{\veps, n}(t) \ge  &  \int_{\sfX} \sum_{j,k=1}^2 \partial_{jk}^2 \bell_\veps \bra{ \partial_t^\parallel \eta^j_t  \partial_t^\parallel \eta^k_t + \ang{ \nabla^\parallel \eta^j_t, \nabla ^\parallel\eta^k_t} } w_n \de \mm \\
& + \int_{\sfX} \sum_{j=1}^2  \partial_j \bell_{\veps} |\eta_t^j|^{-1}\bra{ | \partial_t ^{\perp}\eta^j_t |^2 +  | \nabla ^{\perp}\eta^j_t |^2} w_n \de \mm \\
& +  \int_{\sfX} \ang{\d \bell_\veps\bra{ |f_t|, |\omega_t| }, \d w_n }  \de \mm. \end{split}\end{equation}
Since the third term above corresponds to the rightmost one in \eqref{eq:convexity-bellman-integrated-approx}, we focus on the other ones. By  \eqref{eq:second-derivative-bellman-approx} and \cref{lem:positive-form}, we obtain the $\mm$-a.e.\ inequality
\[ 
\begin{split}  
& \sum_{j,k=1}^2 \partial_{jk}^2 \bell_\veps \bra{ \partial_t^\parallel \eta^j_t  \partial_t^\parallel \eta^k_t + \ang{ \nabla^\parallel \eta^j_t, \nabla ^\parallel\eta^k_t} } \\
& \quad \ge \delta\bra{[|\omega_t|]_\veps^{2-q} \bra{ |\partial^\parallel_tf_t|^2 + |\nabla^\parallel f_t|^2} + [|\omega_t|]_\veps^{q-2} \bra{ |\partial^\parallel_t  \omega_t|^2 + |\nabla^\parallel \omega_t|^2} }.
\end{split} 
\]
By \eqref{eq:first-derivative-bellman-approx} it follows that
\[ 
\begin{split} 
& \sum_{j=1}^2  \partial_j \bell_{\veps} |\eta_t^j|^{-1}\bra{ | \partial_t ^{\perp}\eta^j_t |^2 +  | \nabla ^{\perp}\eta^j_t |^2} \\
& \quad \ge \delta \bra{[|\omega_t|]_\veps^{2-q} \bra{ |\partial^\perp_tf_t|^2 + |\nabla^\perp f_t|^2} + [|\omega_t|]_\veps^{q-2} \bra{ |\partial^\perp_t  \omega_t|^2 + |\nabla^\perp \omega_t|^2} }.
\end{split}
\]
Summing the two inequalities and using \eqref{eq:inequality-convexity-multiplicative}, we deduce that the first two integrals in \eqref{eq:inequality-temp-main-proof} are bounded from below by
\[ 
\delta \int_{\sfX}  | \overline{\nabla} f_t | |\overline{\nabla} \omega_t|  w_n \de \mm,
\]
hence \eqref{eq:convexity-bellman-integrated-approx} is proved.
\end{proof}

\section{Proof of \cref{cor:lusin}}\label{sec:lusin}

In this section we briefly recall how a Lusin-type approximation for Sobolev by Lipschitz functions on $\RCD(K, \infty)$ spaces has been obtained in \cite{ambrosio_lusin-type_2018} and show that \cref{thm:main} allows to extend the result from $p=2$ to general $p\in (1, \infty)$. In fact, \cite[Theorem 4.1]{ambrosio_lusin-type_2018} already ensures that, given $\alpha \in (1,2]$ there exists $c = c(\alpha, K)<\infty$  such that the following holds: for any $f \in W^{1,2}(\sfX)$ there exists an $\mm$-negligible set $N\subset \sfX$ such that, for every $x,y\in \sfX \setminus N$ with $\dist(x,y)\leq 1/ K_-^2$ (to be understood as $+\infty$ if $K_-=0$), 
\[
|f(x)-f(y)| \leq c \dist(x,y) (g(x)+g(y)),
\]
where
\begin{equation}\label{eq:g-abt} 
g(x) = \sup_{s>0} (\sfh_s|\d f|^\alpha(x))^{1/\alpha} + \sup_{s>0} \sfh_s | \sqrt{- \Delta} f| (x).
\end{equation}
This bound is already sufficient to cover the case $K \ge 0$. Indeed, let $p \in (1, \infty)$ and assume that $f \in W^{1,p}(\sfX)$. Then, $|\d f |^{\alpha} \in L^{p/\alpha}(\sfX)$ and by \cref{thm:main}, $\sqrt{- \Delta} f \in L^p(\sfX)$. Thus if we restrict to $\alpha < p$, by the maximal theorem for semigroups, see e.g.\ \cite[p.\ 65]{stein1970topics}, we deduce that $g \in L^p(\sfX, \mm)$.

If $K <0$, we need to slightly modify the argument leading to \eqref{eq:g-abt} to obtain instead
\begin{equation}\label{eq:g-abt2} 
g(x) = \sup_{s>0} (\sfh_s|\d f|^\alpha(x))^{1/\alpha} + \sup_{s>0} \sfh_t | \sqrt{K_- - \Delta} f| (x),
\end{equation}
for which we can repeat the argument above, using \cref{thm:main} in the case $K<0$.

To prove \eqref{eq:g-abt2}, using that $f \in W^{1,p}(\sfX) \subseteq D_p( K_- - \Delta)$ (by \cref{thm:main}) and repeating the arguments in \cite[Theorem 3.6, Remark 3.8]{ambrosio_lusin-type_2018} in the $\RCD$ setting, we obtain that
\begin{equation}\label{eq:fractional}
\abs{ e^{-K_- t} \sfh_t f(x) - f (x) } \le \frac{ 4 \sqrt{t}}{\sqrt{\pi}} \sup_{s>0} \abs{ \sfh_s \sqrt{K_- - \Delta} f}(x), \qquad \forall t > 0,\,\text{for $\mm$-a.e. $x \in \sfX$.} 
\end{equation}
Moreover, by repeating the proof of \cite[Proposition 3.5]{ambrosio_lusin-type_2018}, we have, for every $\alpha\in (1, 2]$ and $t>0$,
\begin{equation}\label{eq:wang} 
\abs{ \sfh_tf(x_0) - \sfh_t f(x_1) } \leq \dist(x_0,x_1) e^{-K t} \exp\cur{ \frac{ \dist(x_0, x_1)^2}{2 \sigma_K(t)(\alpha-1)} }( \sfh_t | \d f|^\alpha (x_0) )^{1/\alpha}, \quad \forall x_1 \in \sfX
\end{equation}
for $\mm$-a.e.\ $x_0 \in \sfX$, where $\sigma_K(t) = K^{-1}(e^{2Kt} -1)$. Combining \eqref{eq:fractional} with \eqref{eq:wang} and choosing  $t = \sqrt{\dist(x_0, x_1)} \le 1/K_-$, we eventually obtain \eqref{eq:g-abt2}.

\section{Examples}\label{sec:examples}

\subsection{Bessel operators}\label{sec:bessel}

As first example, we will discuss the case of Bessel operators, showing that they can be seen as ``metric'' Laplacians on suitable $\RCD(K,\infty)$ spaces. This will allow to recover, as a byproduct of our main results, the well-known boundedness of the Riesz transform in this setting.

\subsubsection*{Setting}


Given $n \in \N_+$ and a multi-index $\alpha = (\alpha_1,\ldots,\alpha_n) \in [0,+\infty)^n$, we define the Bessel operator as
\begin{equation}\label{eq:bessel-operator}
B_\alpha := -\sum_{j=1}^n \Big(\partial_j^2 + \frac{2\alpha_j}{x_j}\partial_j\Big), \qquad D(B_\alpha) := \{f \in C^2_{\sfb\sfs}(\overline{\sfM}) \,:\, \partial_\sfn f = 0 \textrm{ on }\partial\sfM\}
\end{equation}
where $\sfM := (0,+\infty)^n$ and $\partial_\sfn$ denotes the normal derivative. The choice of $D(B_\alpha)$ is motivated by the fact that when $\alpha=0$, $B_0 = -\Delta$ and for the self-adjoint extension of $\Delta$ to coincide with the metric Laplacian we need to impose Neumann boundary conditions. The operator $B_\alpha$ is non-negative and symmetric on $D(B_\alpha)$ with respect to the inner product of $L^2(\sfM,\mm_\alpha)$, where
\[
\mm_\alpha(\de x) := x^{2\alpha}\de x = x_1^{2\alpha_1} \cdots x_n^{2\alpha_n}\de x_1\cdots \de x_n,
\]
and extends uniquely to a self-adjoint operator (that will still be denoted by $B_\alpha$) on $L^2(\sfM,\mm_\alpha)$. By definition, the Sobolev space $\mathbb{W}^{1,2}(\sfM,\mm_\alpha)$ is the closure of $C^\infty_c(\sfM)$ with respect to the Sobolev norm
\[
\|f\|^2_{\mathbb{W}^{1,2}(\sfM,\mm_\alpha)} := \|f\|^2_{L^2(\sfM,\mm_\alpha)} + \||\nabla f|\|^2_{L^2(\sfM,\mm_\alpha)}
\]
and can be equivalently described as
\begin{equation}\label{eq:bessel-sobolev-space}
\mathbb{W}^{1,2}(\sfM,\mm_\alpha) = \{ f \in \mathbb{W}^{1,2}_{{\sf loc}}(\R_+^n) \,:\, f,\partial_j f \in L^2(\R_+^n,\mm_\alpha),\, 1 \leq j \leq n\},
\end{equation}
where $\partial_j$ denotes the distributional $j$-th partial derivative.

\subsubsection*{Fitting the $\RCD$ theory}

Now let us discuss how the setting presented so far can be recast using the language of $\RCD$ spaces and how the Bessel operator $B_\alpha$ can be seen as a ``metric'' Laplace operator.

First of all, since $\sfM$ is not complete, we need to consider its completion $\overline{\sfM} = [0,+\infty)^n$. We endow it with the Euclidean distance $\d_e$ and with the extension of the measure $\mm_\alpha$ to $\overline{\sfM}$; for sake of simplicity, such extension will still be denoted by $\mm_\alpha$. As a first step, we investigate the Sobolev space built over $(\overline{\sfM},\sfd_e,\mm_\alpha)$ and prove that the metric measure space $(\overline{\sfM},\dist_e,\mm_\alpha)$ is infinitesimally Hilbertian.

\begin{proposition}\label{prop:sobolev-bessel}
For all $f \in \Lip(\overline{\sfM},\dist_e)$ it holds $|Df|=|\d f|_e$ $\mm_\alpha$-a.e.\ in $\overline{\sfM}$. 

As a consequence:
\begin{itemize}
\item[(a)] $W^{1,2}(\overline{\sfM},\dist_e,\mm_\alpha) = \mathbb{W}^{1,2}(\sfM,\mm_\alpha)$, the latter being defined as in \eqref{eq:bessel-sobolev-space};
\item[(b)] the Cheeger energy $E$ built on $(\overline{\sfM},\dist_e,\mm_\alpha)$ is a quadratic form;
\item[(c)] the metric differential and the pointwise norm $|\cdot| : L^2(T^*\overline{\sfM}) \to L^2(\overline{\sfM},\mm_\alpha)$ can be naturally identified with the Euclidean differential and norm $|\cdot|_e$, respectively.
\end{itemize}
\end{proposition}

\begin{proof}
The space $\Lip(\overline{\sfM},\dist_e)$ is dense in $\mathbb{W}^{1,2}(\sfM,\mm_\alpha)$ by the regularity of $\partial\sfM$ (see e.g.\ \cite[Theorem 7.2]{kufner1980weighted}) and, as proved in \cite{ambrosio2013density}, it is also dense in $W^{1,2}(\overline{\sfM}, \dist_e, \mm_\alpha)$. Therefore, if we prove that $|Df|=|\d f|_e$ $\mm_\alpha$-a.e.\ in $\overline{\sfM}$ for all $f \in \Lip(\overline{\sfM},\dist_e)$, where $|\d f|_e$ denotes the Euclidean norm of the (Euclidean) differential, then we have
\[
W^{1,2}(\overline{\sfM}, \dist_e, \mm_\alpha) = \overline{\Lip(\overline{\sfM},\dist_e)}^{W^{1,2}(\overline{\sfM},\dist_e,\mm_\alpha)} = \overline{\Lip(\overline{\sfM},\dist_e)}^{\mathbb{W}^{1,2}(\sfM,\mm_\alpha)} = \mathbb{W}^{1,2}(\sfM,\mm_\alpha).
\]
In order to prove that $|Df|=|\d f|_e$ $\mm_\alpha$-a.e.\ in $\overline{\sfM}$, let $f \in \Lip(\overline{\sfM},\dist_e)$ and, for any $\eps>0$, let $\sfN_\eps := [\eps,\eps^{-1}]^n$. By \cite[Theorem 4.20]{AmbrosioGigliSavare11-2} and the fact that $\Lip(\overline{\sfM},\dist_e) \subseteq W^{1,2}(\overline{\sfM},\dist_e,\mm_\alpha)$, the function $g_\eps := f|_{\sfN_\eps}$ is Sobolev in $(\sfN_\eps,\dist_e|_{\sfN_\eps \times \sfN_\eps},\mm_\alpha|_{\sfN_\eps})$ and, if we denote by $|D\cdot|_\eps$ the minimal weak upper gradient determined by the metric measure space $(\sfN_\eps,\dist_e|_{\sfN_\eps \times \sfN_\eps},\mm_\alpha|_{\sfN_\eps})$, then $|Dg_\eps|_\eps = |Df|$ $\mm_\alpha$-a.e. Let us further invoke \cite[Theorem 3.6]{ambrosio2015riemannian} (see also \cite[Remark 5.18]{AGS14}): since the density of $\mm_\alpha|_{\sfN_\eps}$ with respect to the Lebesgue measure $\Leb{n}$ is globally bounded away from 0 and $\infty$ in $\sfN_\eps$, this ensures that the classes of $\mm_\alpha|_{\sfN_\eps}$-test plans and $\Leb{n}$-test plans coincide, so that they induce the same minimal weak upper gradient. In other words, 
\[
S^2(\sfN_\eps,\dist_e|_{\sfN_\eps \times \sfN_\eps},\mm_\alpha|_{\sfN_\eps}) = S^2(\sfN_\eps,\dist_e|_{\sfN_\eps \times \sfN_\eps},\Leb{n}|_{\sfN_\eps})
\]
and
\[
|D\varphi|_\eps = |D\varphi|_{\eps,\Leb{n}}, \qquad \forall \varphi \in S^2(\sfN_\eps,\dist_e|_{\sfN_\eps \times \sfN_\eps},\Leb{n}|_{\sfN_\eps}),
\]
where $|D\varphi|_{\eps,\Leb{n}}$ denotes the minimal weak upper gradient associated with the metric measure space $(\sfN_\eps,\dist_e|_{\sfN_\eps \times \sfN_\eps},\Leb{n}|_{\sfN_\eps})$. Now, by consistency with the smooth case we have $|D\varphi|_{\eps,\Leb{n}} = |\d\varphi|_e$ and this proves that $|Df| = |\d g_\eps|_e = |\d f|_e$ $\mm_\alpha$-a.e.\ in $\sfN_\eps$, for every $\eps>0$. By the arbitrariness of $\eps$ and the fact that $\partial\sfM$ is $\mm_\alpha$-negligible, we conclude that $|Df| = |\d f|_e$ $\mm_\alpha$-a.e.\ in $\overline{\sfM}$.


By density, the fact that $|Df|=|\d f|_e$ $\mm_\alpha$-a.e.\ in $\overline{\sfM}$ extends from $f \in \Lip(\overline{\sfM},\dist_e)$ to $f \in W^{1,2}(\overline{\sfM},\dist_e,\mm_\alpha)$ and this allows us to deduce that
\[
E(f) = \int_\sfM |\d f|^2_e \de\mm_\alpha, \qquad \forall f \in W^{1,2}(\overline{\sfM},\dist_e,\mm_\alpha)
\]
and 
\[
\|f\|^2_{W^{1,2}(\overline{\sfM},\dist_e,\mm_\alpha)} = \|f\|^2_{L^2(\sfM,\mm_\alpha)} + \||\d f|_e\|^2_{L^2(\sfM,\mm_\alpha)},
\]
whence the fact that $E$ is quadratic. As for the final part of the statement, \cite[Remark 2.10]{gigli2018lecture} together with $|Df| = |\d f|_e$ $\mm_\alpha$-a.e.\ in $\sfM$ entails that $L^2(T^*\sfX)$ can be identified with the space of $L^2$-sections of the cotangent bundle via the map that sends the metric differential of a function $f$ to the distributional differential of $f$. This implies, in particular, the identification of the norms as well.
\end{proof}

As a second step, we connect the metric Laplacian $\Delta_\alpha$, arising from $(\overline{\sfM},\dist_e,\mm_\alpha)$, with the Bessel operator $B_\alpha$.

\begin{proposition}\label{prop:bessel-laplacian}
The ``metric'' Bessel operator $\cB_\alpha := -\Delta_\alpha$ coincides with $B_\alpha$.
\end{proposition}

\begin{proof}
As by the previous proposition we know that \[D(B_\alpha) \subseteq \mathbb{W}^{1,2}(\sfM,\mm_\alpha) = W^{1,2}(\overline{\sfM},\dist_e,\mm_\alpha),\] by the density of $\Lip(\overline{\sfM},\dist_e)$ in $W^{1,2}(\overline{\sfM},\dist_e,\mm_\alpha)$ (see \cite{ambrosio2013density}) it is sufficient to show that for any $f \in D(B_\alpha)$ it holds
\begin{equation}\label{eq:bessel-ibp}
\int_\sfM \langle \d f,\d \phi\rangle\de\mm_\alpha = \int_\sfM B_\alpha(f) \phi\de\mm_\alpha, \qquad \forall \phi \in \Lip(\overline{\sfM},\dist_e).
\end{equation}
Indeed, this means that $f \in D(\cB_\alpha)$ and $\cB_\alpha f = B_\alpha f$, namely $\cB_\alpha$ is a (non-negative) self-adjoint extension of $B_\alpha$, which is non-negative and self-adjoint as well. The conclusion now follows observing that on the one hand the inclusion $D(B_\alpha) \subseteq D(\cB_\alpha)$ entails $D(\cB_\alpha^*) \subseteq D(B_\alpha^*)$, while on the other hand by self-adjointness $D(B_\alpha^*) = D(B_\alpha) \subseteq D(\cB_\alpha) = D(\cB_\alpha^*)$, whence $D(\cB_\alpha^*) = D(B_\alpha^*)$. By reflexivity this implies $D(\cB_\alpha) = D(B_\alpha)$ and thus $(\cB_\alpha, D(\cB_\alpha)) = (B_\alpha, D(B_\alpha))$.

As for the proof of \eqref{eq:bessel-ibp}, let us first consider the case $n=1$. Fix $s \in (0,+\infty)$, $\phi \in \Lip([0,+\infty),\dist_e)$ and $f \in C^2_{\sfb\sfs}([0,+\infty))$ with $\partial_\sfn f(0) = f'(0)=0$; for sake of brevity, let us also introduce $h^s(t) := (\partial_t f)(t)t^{2s}\phi(t)$ for $t \geq 0$ and notice that
\[
(\partial_t h^s)(t) = (\partial_t f)(t)t^{2s}(\partial_t \phi)(t) + \partial_t\big((\partial_t f)(t)t^{2s}\big)\phi(t),
\]
so that
\[
(\partial_t f)(t)(\partial_t \phi)(t)t^{2s} = (\partial_t h^s)(t) + B_s(f)(t)\phi(t)t^{2s}
\]
in the sense of distributions on $(0,+\infty)$. This implies
\[
\begin{split}
\int_{\sfM^1} \langle \d f,\d \phi\rangle \de\mm_s & = \int_0^\infty B_s(f)\phi t^{2s}\de t + \int_0^\infty \partial_t h^s \de t \\
& = \int_{\sfM^1} B_s(f)\phi\de\mm_s - h^s(0) = \int_{\sfM^1} B_s(f)\phi\de\mm_s, 
\end{split}
\]
whence the conclusion. For the general case, 
fix $\alpha \in (0,+\infty)^n$, $\phi \in \Lip(\overline{\sfM},\dist_e)$, $f \in D(B_\alpha)$, and set for sake of brevity $\hat \alpha_j := (\alpha_1,\ldots,\alpha_{j-1},\alpha_{j+1},\ldots,\alpha_n)$ (the definition of $\hat x_j$ is similar). Then observe that for every $j=1,\ldots,n$, $\partial_\sfn f = \partial_j f$ on $\overline{\sfM} \cap \{x_j=0\}$, so that for every $x_1,\ldots,x_{j-1},x_{j+1},\ldots,x_n$ fixed, the map $x_j \mapsto f(x_1,\ldots,x_n)$ is $C^2_{\sfb\sfs}([0,+\infty))$ with vanishing derivative in 0. Hence, by Fubini-Tonelli and by the case $n=1$ we have
\[
\begin{split}
\int_\sfM \langle\d f,\d\phi\rangle\de\mm_\alpha & = \sum_{j=1}^n \int_\sfM\partial_j f \partial_j\phi\, \de\mm_\alpha = \sum_{j=1}^n \int_{[0,+\infty)^{n-1}}\bigg(\int_0^\infty \partial_j f\partial_j\phi \,\de\mm_{\alpha_j}(x_j)\bigg)\de\mm_{\hat \alpha_j}(\hat x_j) \\
& = \sum_{j=1}^n \int_{[0,+\infty)^{n-1}}\bigg(\int_0^\infty B_{\alpha_j}(f)\phi\,\de\mm_{\alpha_j}(x_j)\bigg)\de\mm_{\hat\alpha_j}(\hat x_j) \\
& = \sum_{j=1}^n \int_{\overline{\sfM}} B_{\alpha_j}(f)\phi\,\de\mm_\alpha = \int_{\overline{\sfM}} B_\alpha(f)\phi\,\de\mm_\alpha,
\end{split}
\]
where the last identity is due to the fact that $B_\alpha$ is the sum of the $n$ operators $B_{\alpha_1},\ldots,B_{\alpha_n}$ acting separately on the coordinates $x_1,\ldots,x_n$.
\end{proof}

Finally, let us prove that $(\overline{\sfM},\dist_e,\mm_\alpha)$ is an $\RCD(K,\infty)$ space, for some $K \in \R$.

\begin{proposition}\label{prop:bessel-rcd}
Let $\alpha \in [0,+\infty)^n$. Then $(\overline{\sfM},\dist_e,\mm_\alpha)$ is an $\RCD(0,\infty)$ space.
\end{proposition}

\begin{proof}
By the tensorization property of the $\RCD$ condition (see \cite[Theorem 7.6]{ambrosio2015riemannian}) and by the product structure of $\overline{\sfM}$, $\dist_e$, and $\mm_\alpha$, it is sufficient to prove the statement for $n=1$, namely: $([0,+\infty),\dist_e,t^{2\alpha}\de t)$ is an $\RCD(0,\infty)$ space. To this end, note that $t^{2\alpha} = e^{-V}$ with $V=-2\alpha\log t$ and 
\[
{\rm Ent}_{e^{-V}\Leb{1}}(\nu) = {\rm Ent}_{\Leb{1}}(\nu) + \int_{\overline{\sfM}} V\,\de\nu, \qquad \forall \nu \in \Probabilities{\overline{\sfM}}.
\]
Since $V'' \geq 0$ and $([0,+\infty),\dist_e,\de t)$ is flat, by \cite[Theorem 1.3]{sturm2005convex} we conclude that ${\rm Ent}_{e^{-V}\Leb{1}}$ is geodesically convex and thus the $\CD(0,\infty)$ condition is satisfied. As $([0,+\infty),\dist_e,t^{2\alpha}\de t)$ is also infinitesimally Hilbertian by Proposition \ref{prop:sobolev-bessel}, we conclude that the space is in fact $\RCD(0,\infty)$.
\end{proof}

By the identification of the ``metric'' differential with the Euclidean one and the ``metric'' Bessel operator with $B_\alpha$, we deduce that the ``metric'' Riesz transform $\cR_\alpha^a := \d (a + \cB_\alpha)^{-1/2}$ can be identified with the Bessel-Riesz transform $\mathcal{R}_\alpha^a := \d(a + B_\alpha)^{-1/2}$ and the boundedness of the latter is ensured by Theorem \ref{thm:main} together with Proposition \ref{prop:bessel-rcd}, so that in conclusion we have the following

\begin{corollary}\label{cor:main-bessel}
Let $\alpha \in [0,+\infty)^n$. Let $B_\alpha$ denote the (unique) self-adjoint extension of \eqref{eq:bessel-operator} on the core $D(B_\alpha)$ defined in \eqref{eq:bessel-operator}. Then, for every $p \in (1,\infty)$ and $f \in \mathbb{W}^{1,2}(\sfM,\mm_\alpha)$ it holds
\[
c_p^{-1}\nor{\sqrt{-\Delta} f}_p \leq \nor{\d f}_p \leq c_p \nor{\sqrt{-\Delta} f}_p,
\]
where $c_p = 16\max\{p,p/(p-1)\}$.
\end{corollary}

\subsection{Laguerre operators}\label{sec:laguerre}

As a second example, let us consider Laguerre operators. For the definitions and notations we follow \cite{mauceri2015riesz}, where dimension-free $L^p$-bounds on the Riesz transforms are obtained. As concerns the interpretation of Laguerre operators as differential operators on weighted manifolds with boundary, some of the following considerations can also be found in \cite{crasmareanu2015weighted}, where other examples are discussed.

\subsubsection*{Setting}

Given $n \in \N_+$ and a multi-index $\alpha = (\alpha_1,\ldots,\alpha_n) \in (-1,+\infty)^n$, the Laguerre operator of type $\alpha$ is defined as
\begin{equation}\label{eq:laguerre-operator}
L_\alpha := -\sum_{j=1}^n \Big( x_j\partial_j^2 + (\alpha_j+1-x_j)\partial_j\Big)
\end{equation}
on the domain $C_c^\infty(\sfM)$, where $\sfM := (0,+\infty)^n$. It is non-negative and symmetric on such domain with respect to the inner product of $L^2(\sfM,\mm_\alpha)$, where
\[
\mm_\alpha := \prod_{j=1}^n \frac{x_j^{\alpha_j}e^{-x_j}}{\Gamma(\alpha_j+1)}\de x_j.
\]
If $\alpha \in [1,+\infty)^n$, then $L_\alpha$ is also essentially self-adjoint on $L^2(\sfM,\mm_\alpha)$, but if $\alpha_j < 1$ for some $j$, then $L_\alpha$ has several self-adjoint realizations, depending on the boundary conditions at $x_j=0$. In order to select a specific extension, let us define the Laguerre polynomial of type $\alpha$ and multidegree $k$ on $\sfM$ as
\[
P_k^\alpha(x) = P_{k_1}^{\alpha_1}(x_1) \cdots P_{k_n}^{\alpha_n}(x_n), \qquad P_{k_i}^{\alpha_i}(x_i) := \frac{1}{k!}e^{x_i}x_i^{-\alpha_i}\frac{\de^{k_i}}{\de x_i^{k_i}}(e^{-x_i}x_i^{k_i+\alpha_i}).
\]
These polynomials are orthogonal in $L^2(\sfM,\mm_\alpha)$, so that if we denote by $\ell_k^\alpha := \|L_k^\alpha\|_{L^2(\sfM,\mm_\alpha)}^{-1} L_k^\alpha$ their normalizations, the family $\{\ell_k^\alpha\}_{k \in \N^n}$ is an orthonormal basis of $L^2(\sfM,\mm_\alpha)$ consisting of eigenfunctions of $L_\alpha$. Namely, $L_\alpha \ell_k^\alpha = |k| \ell_k^\alpha$, for all $k \in \N^n$, where $|k| := k_1 + \cdots + k_n$ denotes the length of the multi-index. This allows to define
\begin{equation}\label{eq:laguerre-operator2}
L_\alpha := \sum_{j=0}^\infty j\mathscr{P}_j, \qquad \mathscr{P}_j f := \sum_{|k| = j} \langle f,\ell_k^\alpha \rangle_{L^2(\sfM,\mm_\alpha)}\ell_k^\alpha
\end{equation}
with domain
\[
D(L_\alpha) := \Big\{f \in L^2(\sfM,\mm_\alpha) \,:\, \sum_{j=0}^\infty j^2 \|\mathscr{P}_j f\|^2_{L^2(\sfM,\mm_\alpha)} < +\infty \Big\}.
\]
If $\alpha \in [1,+\infty)^n$ this coincides with the spectral resolution of the closure of $L_\alpha$ as defined in \eqref{eq:laguerre-operator}, whereas for all other values of $\alpha \in (-1,+\infty)^n$ this will be the self-adjoint extension of \eqref{eq:laguerre-operator} we shall consider.

If we introduce the Laguerre partial derivatives $\delta_j := \sqrt{x_j}\partial_j$, $1 \leq j \leq n$, then we can also define the Sobolev space $\mathbb{W}^{1,2}(\sfM,\mm_\alpha)$ as the closure of $C^\infty_c(\sfM)$ with respect to the norm
\[
\|f\|^2_{\mathbb{W}^{1,2}(\sfM,\mm_\alpha)} := \|f\|^2_{L^2(\sfM,\mm_\alpha)} + \sum_{j=1}^n \|\delta_j f\|^2_{L^2(\sfM,\mm_\alpha)},
\]
whence
\begin{equation}\label{eq:laguerre-sobolev-space}
\mathbb{W}^{1,2}(\sfM,\mm_\alpha) = \{ f \in \mathbb{W}^{1,2}_{{\sf loc}}(\R_+^n) \,:\, f,\delta_j f \in L^2(\R_+^n,\mm_\alpha),\, 1 \leq j \leq n\}.
\end{equation}

\subsubsection*{Fitting the $\RCD$ theory}

After this premise, let us now explain how $L_\alpha$ can be seen as ``metric'' Laplace operator on some $\RCD(K,\infty)$ space. To this end, let us first introduce the Riemannian metric naturally associated with $L_\alpha$: this is given by $\sfg := \sum_{j=1}^n x_j^{-1}\de x_j \otimes \de x_j$ and a straightforward calculation shows that the gradient and differential associated with the Riemannian manifold $(\sfM,\sfg)$ are
\begin{equation}\label{eq:gradient-differential}
\nabla = (\sqrt{x_1}\partial_1,\cdots,\sqrt{x_n}\partial_n) \qquad \textrm{and} \qquad \d = \Big(\frac{1}{\sqrt{x_1}}\partial_1, \cdots, \frac{1}{\sqrt{x_n}}\partial_n \Big),
\end{equation}
so that for any $C^1$ function $f$ it holds
\begin{equation}\label{eq:metric-partial_deriv}
|\nabla f|_\sfg^2 = |\d f|^2_{\sfg^{-1}} = \sum_{j=1}^n |\delta_j f|^2.
\end{equation}
Moreover, $\sfg$ induces a volume form ${\rm vol}_\sfg$ and a distance $\dist_\sfM$. The former is given by
\begin{equation}\label{eq:laguerre-volume}
{\rm vol}_\sfg(\de x) = \Big(\prod_{j=1}^n x_j\Big)^{-1}\de x
\end{equation}
while the latter is
\begin{equation}\label{eq:distance}
\dist_\sfM(x,y)^2 = \sum_{j=1}^n \dist_{\sfM^1}(x_j,y_j)^2 = 4\sum_{j=1}^n |\sqrt{x_j}-\sqrt{y_j}|^2,
\end{equation}
where $(\sfM^1,\sfg^1) = ((0,+\infty),t^{-1}\de t \otimes \de t)$. Indeed, $(\sfM,\sfg)$ is the product of $n$ copies of $(\sfM^1,\sfg^1)$, which explains the first identity, while geodesics in $(\sfM^1,\sfg^1)$ can be computed observing that the Christoffel's symbol of $\sfg^1$ is $\Gamma(t)=1/2t$, whence the following geodesic equation
\[
\eta''_t = \frac{(\eta'_t)^2}{2\eta_t} = \frac12 |\eta'_t|^2_{\sfg^1_{\eta_t}}
\]
for a curve $\eta : [0,1] \to \sfM^1$. The solutions to this equation are of the form $\eta_t = a(t+b)^2$ for $a \neq 0$, $b \in \R$, and solving $\eta_0 = ab^2$, $\eta_1 = a(b+1)^2$ yields $a=|\sqrt{\eta_1}-\sqrt{\eta_0}|^2$, whence $\sfd_{\sfM^1}(\eta_1,\eta_0)^2 = 4a = 4|\sqrt{\eta_1}-\sqrt{\eta_0}|^2$. 

As $(\sfM,\sfd_\sfM)$ is not complete, we need to consider its completion $\overline{\sfM} = [0,+\infty)^n$ endowed with the extension of $\dist_\sfM$ to $\overline{\sfM} \times \overline{\sfM}$ (which can be naturally done via \eqref{eq:distance}). Note also that the measure $\mm_\alpha$ previously defined is locally finite on $\overline{\sfM}$ if and only if $\alpha \in (-1,+\infty)^n$, consistently with the restrictions on the parameter, and can be naturally extended to a Radon measure on $\overline{\sfM}$. As in the case of Bessel operators, our investigation starts with the Sobolev space and the infinitesimal Hilbertianity of $(\overline{\sfM},\dist_\sfM,\mm_\alpha)$. 


\begin{proposition}\label{prop:sobolev-laguerre}
For all $f \in \Lip(\overline{\sfM},\dist_\sfM)$ it holds $|Df|=|\d f|_{\sfg^{-1}}$ $\mm_\alpha$-a.e.\ in $\overline{\sfM}$. 

As a consequence:
\begin{itemize}
\item[(a)] $W^{1,2}(\overline{\sfM},\dist_\sfM,\mm_\alpha) \subseteq \mathbb{W}^{1,2}(\sfM,\mm_\alpha)$, the latter being defined as in \eqref{eq:laguerre-sobolev-space};
\item[(b)] the Cheeger energy $E$ built on $(\overline{\sfM},\dist_\sfM,\mm_\alpha)$ is a quadratic form;
\item[(c)] the metric differential and the pointwise norm $|\cdot| : L^2(T^*\overline{\sfM}) \to L^2(\overline{\sfM},\mm_\alpha)$ can be naturally identified with the Riemannian differential \eqref{eq:gradient-differential} and $|\cdot|_{\sfg^{-1}}$, respectively.
\end{itemize}
\end{proposition}

\begin{proof}
As concerns the identity $|Df|=|\d f|_{\sfg^{-1}}$, the proof is analogous to the one explained in Proposition \ref{prop:sobolev-bessel}. Indeed, leveraging on \cite[Theorem 4.20]{AmbrosioGigliSavare11-2} and \cite[Theorem 3.6]{ambrosio2015riemannian} one can show that, for any $f \in \Lip(\overline{\sfM},\dist_\sfM)$ and $\eps>0$, it holds $|Df| = |Df|_{\eps,{\rm vol}_\sfg}$ $\mm_\alpha$-a.e.\ in $\sfN_\eps := [\eps,\eps^{-1}]^n$, where $|Df|_{\eps,{\rm vol}_\sfg}$ denotes the minimal weak upper gradient associated with the metric measure space $(\sfN_\eps,\dist_\sfM|_{\sfN_\eps \times \sfN_\eps},{\rm vol}_\sfg|_{\sfN_\eps})$, ${\rm vol}_\sfg$ being defined as in \eqref{eq:laguerre-volume}. 

Moreover, by consistency with the smooth case $|Df|_{\eps,{\rm vol}_\sfg}$ coincides with the $|\cdot|_{\sfg^{-1}}$-norm of the differential associated with the Riemannian manifold $(\sfN_\eps,\dist_\sfM|_{\sfN_\eps \times \sfN_\eps},{\rm vol}_\sfg|_{\sfN_\eps})$, namely
\[
|Df|^2_{\eps,{\rm vol}_\sfg} = |\d f|^2_{\sfg^{-1}} \stackrel{\eqref{eq:metric-partial_deriv}}{=} \sum_{j=1}^n|\delta_j f|^2, \qquad \mm_\alpha\textrm{-a.e. in } \sfN_\eps.
\]
Recall indeed that, although being called ``gradient'', $|Df|$ is rather a cotangent object, whence $|\d f|_{\sfg^{-1}}$ instead of $|\d f|_\sfg$. Therefore $|Df| = |\d (f|_{\sfN_\eps})|_{\sfg^{-1}}$ $\mm_\alpha$-a.e.\ in $\sfN_\eps$, for every $\eps>0$, and since $f \in \Lip(\overline{\sfM},\dist_\sfM) \subseteq W^{1,2}(\overline{\sfM},\dist_\sfM,\mm_\alpha)$, this means that $f|_{\sfN_\eps} \in \mathbb{W}^{1,2}(\sfN_\eps,\mm_\alpha|_{\sfN_\eps})$ for every $\eps>0$. Namely, $f \in \mathbb{W}^{1,2}_{{\sf loc}}(\R^n_+)$ with $\d f := \d f|_{\sfN_\eps}$, the definition being meaningful by locality of the (Riemannian) differential. In conclusion, $|Df| = |\d f|_{\sfg^{-1}}$ $\mm_\alpha$-a.e.\ in $\sfM$ and $f \in \mathbb{W}^{1,2}(\sfM,\mm_\alpha)$, so that $\Lip(\overline{\sfM},\dist_\sfM) \subseteq \mathbb{W}^{1,2}(\sfM,\mm_\alpha)$. By taking the closure of $\Lip(\overline{\sfM},\dist_\sfM)$ in the $W^{1,2}$- and $\mathbb{W}^{1,2}$-norms and using again the density of $\Lip(\overline{\sfM},\dist_\sfM)$ in $W^{1,2}(\overline{\sfM}, \dist_\sfM, \mm_\alpha)$, see \cite{ambrosio2013density}, we also obtain
\[
W^{1,2}(\overline{\sfM}, \dist_\sfM, \mm_\alpha) = \overline{\Lip(\overline{\sfM},\dist_\sfM)}^{W^{1,2}(\overline{\sfM}, \dist_\sfM, \mm_\alpha)} = \overline{\Lip(\overline{\sfM},\dist_\sfM)}^{\mathbb{W}^{1,2}(\sfM,\mm_\alpha)} \subseteq \mathbb{W}^{1,2}(\sfM,\mm_\alpha),
\]
namely point (a). The fact that $E$ is a quadratic form follows verbatim the approach described in the proof of Proposition \ref{prop:sobolev-bessel} and the same holds for the final part of the statement.
\end{proof}

Unlike Proposition \ref{prop:sobolev-bessel}, we have not investigate the density of $\Lip(\overline{\sfM},\dist_\sfM)$ in the ``classic'' Sobolev space $\mathbb{W}^{1,2}(\sfM,\mm_\alpha)$, as the degeneracy of the metric $\sfg$ on $\partial\sfM$ may lead to technical issues that we prefer to avoid. It is instead sufficient, as next step, to prove that the normalized Laguerre polynomials $\{\ell_k^\alpha\}_{k \in \N^n}$ are Sobolev in the metric sense.

\begin{lemma}\label{lem:laguerre-sobolev}
For every $\alpha \in (-1,+\infty)^n$ and $k \in \N^n$, $\ell_k^\alpha \in W^{1,2}(\overline{\sfM},\dist_\sfM,\mm_\alpha)$.
\end{lemma}

\begin{proof}
Fix $\varphi \in C^\infty_c([0,+\infty))$ with $\varphi \geq 0$, $\supp(\varphi) = [0,2]$, and $\varphi = 1$ on $[0,1]$; for every $m \in \N_+$ define $\psi_m : \overline{\sfM} \to \R$, $\psi_m(x) := \varphi(m^{-1}\dist_\sfM(0,x))$. We observe that $\psi_m \in \Lip(\overline{\sfM},\dist_\sfM)$ and, by the previous proposition,
\begin{equation}\label{eq:lipschitz}
|D\psi| = |\d\psi|_{\sfg^{-1}} \leq \frac{1}{m}\|\varphi'\|_\infty, \qquad \mm_\alpha\textrm{-a.e. in } \sfM.
\end{equation}
From \eqref{eq:distance} it follows that $2\dist_e(x,y) \leq (\|x\|_\infty+\|y\|_\infty) \dist_\sfM(x,y)$, where $\dist_e$ denotes the Euclidean distance, so that
\[
\ell_k^\alpha \in C^\infty(\R^n) \subseteq {\rm Lip}_{{\sf loc}}(\overline{\sfM},\dist_e) \subseteq {\rm Lip}_{{\sf loc}}(\overline{\sfM},\dist_\sfM).
\]
Thus $\psi_m\ell_k^\alpha \in \Lip(\overline{\sfM},\dist_\sfM) \subseteq W^{1,2}(\overline{\sfM},\dist_\sfM,\mm_\alpha)$ for all $m \in \N_+$. From \eqref{eq:metric-partial_deriv} we see that $|\d\ell_k^\alpha|^2_{\sfg^{-1}}$ is a polynomial and therefore, by definition of $\mm_\alpha$, it belongs to $L^1(\sfM,\mm_\alpha)$. From $\d(\psi_m\ell_k^\alpha) = \ell_k^\alpha \d\psi_m + \psi_m\d\ell_k^\alpha$ and \eqref{eq:lipschitz}, by the dominated convergence theorem we deduce that $\d(\psi_m\ell_k^\alpha) \to \d\ell_k^\alpha$ in $L^2(T^*\overline{\sfM})$. Since $\psi_m\ell_k^\alpha \to \ell_k^\alpha$ in $L^2(\sfM,\mm_\alpha) = L^2(\overline{\sfM},\dist_\sfM,\mm_\alpha)$ too and the metric differential $\d$ is closed, we conclude that $\ell_k^\alpha \in W^{1,2}(\overline{\sfM},\dist_\sfM,\mm_\alpha)$.
\end{proof}

We are now ready to study the relationship between $\Delta_\alpha$ and the Laguerre operator $L_\alpha$.

\begin{proposition}\label{prop:laguerre-laplacian}
The ``metric'' Laguerre operator $\cL_\alpha := -\Delta_\alpha$ coincides with $L_\alpha$.
\end{proposition}

\begin{proof}
In order to identify $L_\alpha$ and $\mathcal{L}_\alpha$, it suffices to show that
\begin{equation}\label{eq:identification}
\ell_k^\alpha \in D(\cL_\alpha) \qquad \textrm{and} \qquad \cL_\alpha \ell_k^\alpha = L_\alpha \ell_k^\alpha, \quad \forall k \in \N^n.
\end{equation}
Indeed, if the claim above is true, then for any $f \in D(L_\alpha)$ and $j \in \N$ it holds $\mathscr{P}_j f \in D(\cL_\alpha)$ and $\cL_\alpha \mathscr{P}_j f = L_\alpha \mathscr{P}_j f = j\mathscr{P}_j f$. Hence, if we define $f_m := \sum_{j=0}^m \mathscr{P}_j f$, then we have that $f_m \to f$ in $L^2(\sfM,\mm_\alpha)$, $f_m \in D(\cL_\alpha)$ and $\cL_\alpha f_m = \sum_{j=0}^m j\mathscr{P}_j f \to L_\alpha f$ in $L^2(\sfM,\mm_\alpha)$, so that by closedness of $\cL_\alpha$ we deduce that $f \in D(\cL_\alpha)$ and $\cL_\alpha f = L_\alpha f$. This means that $\cL_\alpha$ is a (non-negative) self-adjoint extension of $L_\alpha$, which is non-negative and self-adjoint as well. The conclusion then follows as already argued in Proposition \ref{prop:bessel-laplacian} for the Bessel operator $B_\alpha$.

Now let us prove \eqref{eq:identification}. Leveraging on the density of $\Lip(\overline{\sfM},\dist_\sfM)$ in $W^{1,2}(\overline{\sfM},\dist_\sfM,\mm_\alpha)$ (see \cite{ambrosio2013density}), on Lemma \ref{lem:laguerre-sobolev} and on the fact that $L_\alpha(\ell_k^\alpha) \in W^{1,2}(\overline{\sfM},\dist_\sfM,\mm_\alpha)$, it is sufficient to show that
\[
\int_\sfM \langle \d\ell_k^\alpha,\d f\rangle_{\sfg^{-1}}\de\mm_\alpha = \int_\sfM L_\alpha(\ell_k^\alpha) f\de\mm_\alpha, \qquad \forall f \in \Lip(\overline{\sfM},\dist_\sfM).
\]
Actually, given the product structure of $\ell_k^\alpha$, $\dist_\sfM$, $\mm_\alpha$, and $L_\alpha$, it is sufficient to check the above identity for $n=1$, i.e.\ on $(\sfM^1,\sfg^1)$. Therefore, let us fix $s \in (-1,+\infty)$, $m \in \N$, and $f \in \Lip([0,+\infty),\dist_{\sfM^1})$ and observe that from \eqref{eq:distance} it holds $\dist_{\sfM^1}(u,v) \leq \eps^{-1}\dist_e(u,v)$ for all $u,v \in [\eps^2,+\infty)$. Setting $h_m^s(t) := (\partial_t\ell_m^s)(t)t^{s+1}e^{-t}f(t)$ for $t \geq 0$, this implies that $h_m^s \in \Lip([\eps^2,+\infty)],\dist_e)$ for all $\eps > 0$, since product of a compactly supported $\dist_e$-Lipschitz function with a $C^\infty$ function, and
\[
(\partial_t h_m^s)(t) = (\partial_t\ell_m^s)(t)t^{s+1}e^{-t}(\partial_t f)(t) + \partial_t\big((\partial_t\ell_m^s)(t)t^{s+1}e^{-t}\big)f(t),
\]
so that
\[
(\partial_t\ell_m^s)(t)(\partial_t f)(t)t^{s+1}e^{-t} = (\partial_t h_m^s)(t) + L_s(\ell_m^s)(t)f(t)t^s e^{-t}
\]
in the sense of distributions on $(0,+\infty)$. Thus, recalling from (the proof of) Lemma \ref{lem:laguerre-sobolev} that $|\d\ell_m^s|_{\sfg^{-1}} \in L^2(\sfM^1,\mm_s)$ to justify the first identity, we have
\[
\begin{split}
\int_\sfM \langle \d\ell_m^s,\d f\rangle_{\sfg^{-1}}\de\mm_s & = \lim_{\eps \downarrow 0}\int_{\eps^2}^\infty \partial_t\ell_m^s \partial f t^{s+1}e^{-t}\de t = \int_0^\infty L_s(\ell_m^s)f t^s e^{-t}\de t + \lim_{\eps \downarrow 0} \int_{\eps^2}^\infty \partial_t h_m^s \de t \\
& = \int_{\sfM^1} L_s(\ell_m^s)f\de\mm_s - \lim_{\eps \downarrow 0} h_m^s(\eps^2) \\
& = \int_{\sfM^1} L_s(\ell_m^s)f\de\mm_s - \lim_{\eps \downarrow 0} \eps^{2(s+1)}C(s,m)\ell_{m-1}^{s+1}(\eps^2) e^{-\eps^2}f(\eps^2) \\
& = \int_{\sfM^1} L_s(\ell_m^s)f\de\mm_s,
\end{split}
\]
where in the second to last identity we used the fact that $h_m^s$ is the product of $t^{s+1}$ with a compactly supported continuous function. The proof of \eqref{eq:identification} is thus concluded.
\end{proof}

The identification of the ``metric'' differential with the Riemannian differential \eqref{eq:gradient-differential} and $\cL_\alpha$ with $L_\alpha$ imply that the ``metric'' Riesz transform $\cR_\alpha^a := \d (a + \cL_\alpha)^{-1/2}$ and the Laguerre-Riesz transform $\mathcal{R}_\alpha^a := \delta(a + L_\alpha)^{-1/2}$, as defined in \cite{mauceri2015riesz}, coincide up to musical isomorphisms. Let us then show that $(\overline{\sfM},\dist_\sfM,\mm_\alpha)$ is an $\RCD(K,\infty)$ space, for some $K \in \R$.

\begin{proposition}
Let $\alpha \in [-1/2,+\infty)^n$. Then $(\overline{\sfM},\dist_\sfM,\mm_\alpha)$ is an $\RCD(1/2,\infty)$ space.
\end{proposition}

\begin{proof}
By the same considerations as in the proof of Proposition \ref{prop:bessel-rcd}, it is not restrictive to consider the case $n=1$. Then consider the map $\phi : [0,+\infty) \to [0,+\infty)$, $\phi(t) := 2\sqrt{t}$, and observe that it is a bijection, $\dist_\sfM(t,t') = \dist_e(\phi(t),\phi(t'))$, and by the change of variable formula it holds
\[
\phi_*\mm_\alpha(\de t) = \frac{1}{2}\frac{t^{2\alpha+1}}{4^\alpha}\frac{e^{-t^2/4}}{\Gamma(\alpha+1)}\de t.
\]
Thus $([0,+\infty),\dist_\sfM,\mm_\alpha)$ and $([0,+\infty),\dist_e,\phi_*\mm_\alpha)$ are isomorphic via $\phi$ as metric measure spaces and by \cite[Proposition 4.12]{Sturm06I} the one is a $\CD(K,\infty)$ if and only if so is the other. In this regard, rewriting $\phi_*\mm_\alpha$ as $e^{-V}\Leb{1}$ with
\[
V = \frac{1}{4}t^2 - (2\alpha+1)\log t + (2\alpha+1)\log 2 + \log\Gamma(\alpha+1)
\]
and recalling that $\alpha \geq -1/2$, we can observe that $V'' \geq 1/2$ and $([0,+\infty),\dist_e,\Leb{1})$ is flat, so that by \cite[Theorem 1.3]{sturm2005convex} we deduce that $([0,+\infty),\dist_\sfM,\mm_\alpha)$ is a $\CD(1/2,\infty)$ space. Since by Proposition \ref{prop:sobolev-laguerre} it is also infinitesimally Hilbertian, the conclusion follows.
\end{proof}

\begin{remark} 
The restriction $\alpha \in [-1/2,+\infty)^n$, consistent with the literature, is not merely technical. Indeed, if $\alpha_j < -1/2$ for some $j=1,...,n$, then $(\overline{\sfM},\dist_\sfM,\mm_\alpha)$ cannot be an $\RCD(K,\infty)$ space, for any $K \in \R$. This can be seen, for instance, by computing explicitly the generalized Ricci tensor in $(\sfM^1,\sfg^1,\mm_\alpha^1)$:
\[
{\rm Ric}_{(\sfM^1,\sfg^1,\mm_\alpha^1)} = \frac12 \sfg^1 + \frac{\alpha+1/2}{2t^2}\de t \otimes \de t.
\]
If $\alpha < -1/2$, it is easily seen that ${\rm Ric}_{(\sfM^1,\sfg^1,\mm_\alpha^1)}$ cannot be lower bounded.
\end{remark}

As a consequence of Theorem \ref{thm:main} and the discussion carried out in the current section, we obtain the following

\begin{corollary}\label{cor:main-laguerre}
Let $\alpha \in [-1/2,+\infty)^n$. Let $L_\alpha$ denote the self-adjoint extension of \eqref{eq:laguerre-operator} given by the spectral resolution \eqref{eq:laguerre-operator2}. Then, for every $p \in (1,\infty)$ and $f \in W^{1,2}(\overline{\sfM},\dist_\sfM,\mm_\alpha)$ it holds
\[
c_p^{-1}\nor{\sqrt{-\Delta} f}_p \leq \nor{\d f}_p \leq c_p \nor{\sqrt{-\Delta} f}_p,
\]
where $c_p = 16\max\{p,p/(p-1)\}$.
\end{corollary}

\begin{remark}
As already mentioned, the $\RCD$ condition is stable with respect to tensorization. Hence, even if we only discussed Bessel and Laguerre operators (since particularly relevant in the literature), more complicated operators can be built by taking products of them.
\end{remark}


\subsection{Log-concave measures}\label{sec:log-concave}

It is well known that any Hilbert space, when endowed with a non-degenerate log-concave probability measure, becomes an $\RCD(0,\infty)$ spaces. For the reader's convenience, let us recall some terminology and key facts. Let $(H, \langle\cdot,\cdot\rangle)$ be a separable Hilbert space and let $\gamma$ be a log-concave probability measure on $H$, i.e.,
\[ \gamma( (1-t)B +tC ) \ge \gamma(B)^{1-t} \gamma(C)^t, \quad \text{for every $B, C \subseteq H$ open and $t \in (0,1)$.}\]
We also assume that $\gamma$ is non-degenerate, i.e., it is not concentrated on a proper closed subspace of $H$. Under these assumptions, it is argued in \cite{ambrosio2009existence} that the quadratic form, initially defined on Fr\'echet-differentiable functions,
\begin{equation}\label{eq:form-log-concave} 
C^1_\sfb(H) \ni f \mapsto  E_\gamma(f) = \int_{H} \abs{\nabla f}^2 \de \gamma
\end{equation}
is closable in $L^2(H, \gamma)$, i.e.\ it can be extended to a Dirichlet form.  Actually, its closure is the Cheeger energy associated to the metric measure space $(H, \abs{\cdot}, \gamma)$, with domain given by the Sobolev space $W^{1,2}(H, \gamma)$.  Moreover, in \cite{ambrosio2009existence} it is also proved that the associated heat semigroup $\sfP$ is linear and satisfies the so called $\mathsf{EVI}$ property, which in turn in \cite{AmbrosioGigliSavare11-2} is shown to be an equivalent characterization of $\RCD$ spaces (in particular, one obtains that the Ricci curvature is non-negative).

Therefore, in this setting our main result \cref{thm:main} yields the following.

\begin{corollary}\label{cor:main-hilbert}
Let $(H, \langle\cdot, \cdot\rangle)$ be a separable Hilbert space and let $\gamma$ be a non-degenerate log-concave probability measure on $H$. Let $\Delta_\gamma$ denote the Laplacian associated to the closure of the form \eqref{eq:form-log-concave} Then, for every $p \in (1, \infty)$, it holds for $f \in W^{1,2}(H, \gamma)$,
\[
c_p^{-1}\nor{\sqrt{-\Delta} f}_p \leq \nor{\d f}_p \leq c_p \nor{\sqrt{-\Delta} f}_p,
\]
where $c_p = 16\max\{p,p/(p-1)\}$.
\end{corollary}

\section*{Declarations}

\subsection*{Author contributions} All authors made equal contributions to this work.

\subsection*{Acknowledgements}   
DT acknowledges support by the HPC Italian National Centre for HPC, Big Data and Quantum Computing - Proposal code CN1 CN00000013, CUP I53C2200 0690001, and by the INdAM-GNAMPA project 2023 ``Teoremi Limite per Dinamiche di Discesa Gradiente Stocastica: Convergenza e Generalizzazione''.

\subsection*{Competing interests} The authors have no competing interests to declare that are relevant to the content of this article.

\subsection*{Data Availability Statement} The sharing of data is not applicable to this article since no datasets were generated or analyzed during the current study.

\printbibliography

@article{esaki2023riesz,
  title={Riesz transforms for {D}irichlet spaces tamed by distributional curvature lower bounds},
  author={Esaki, Syota and Xu, Zi Jian and Kuwae, Kazuhiro},
  journal={arXiv preprint arXiv:2308.12728},
  year={2023}
}

@article{erbar2022tamed,
  title={Tamed spaces--{D}irichlet spaces with distribution-valued {R}icci bounds},
  author={Erbar, Matthias and Rigoni, Chiara and Sturm, Karl-Theodor and Tamanini, Luca},
  journal={Journal de Math{\'e}matiques Pures et Appliqu{\'e}es},
  volume={161},
  pages={1--69},
  year={2022},
  publisher={Elsevier}
}

@article{esaki2023littlewood,
  title={The {L}ittlewood--{P}aley--{S}tein inequality for {D}irichlet space tamed by distributional curvature lower bounds},
  author={Esaki, Syota and Xu, Zi Jian and Kuwae, Kazuhiro},
  journal={Preprint, arXiv:2307.12514},
  year={2023}
}

@book{hall2013quantum,
  title={Quantum theory for mathematicians},
  author={Hall, Brian C},
  year={2013},
  publisher={Springer}
}

@article{dragivcevic2011bilinear,
  title={Bilinear embedding for real elliptic differential operators in divergence form with potentials},
  author={Dragi{\v{c}}evi{\'c}, Oliver and Volberg, Alexander},
  journal={Journal of Functional Analysis},
  volume={261},
  number={10},
  pages={2816--2828},
  year={2011},
  publisher={Elsevier}
}

@book{kufner1980weighted,
  title={Weighted Sobolev Spaces},
  author={Kufner, A.},
  lccn={81165198},
  series={Teubner-Texte zur Mathematik},
  year={1980},
  publisher={Teubner}
}

@article{gigli2018lecture,
  title={Lecture {N}otes {O}n {D}ifferential {C}alculus on $RCD$ {S}paces},
  author={Gigli, Nicola},
  journal={Publications of the Research Institute for Mathematical Sciences},
  volume={54},
  number={4},
  pages={855--918},
  year={2018}
}

@article{mauceri2015riesz,
  title={Riesz transforms and spectral multipliers of the {H}odge--{L}aguerre operator},
  author={Mauceri, Giancarlo and Spinelli, Micol},
  journal={Journal of Functional Analysis},
  volume={269},
  number={11},
  pages={3402--3457},
  year={2015},
  publisher={Elsevier}
}

@article{crasmareanu2015weighted,
  title={Weighted {R}iemannian 1-manifolds for classical orthogonal polynomials and their heat kernel},
  author={Crasmareanu, Mircea},
  journal={Analysis and Mathematical Physics},
  volume={5},
  number={4},
  pages={373--389},
  year={2015},
  publisher={Springer}
}

@article{ambrosio2009existence,
  title={Existence and stability for Fokker--Planck equations with log-concave reference measure},
  author={Ambrosio, Luigi and Savar{\'e}, Giuseppe and Zambotti, Lorenzo},
  journal={Probability theory and related fields},
  volume={145},
  pages={517--564},
  year={2009},
  publisher={Springer}
}

@article{AmbrosioGigliSavare11-2,
	author = {Ambrosio, Luigi and Gigli, Nicola and Savar{\'e}, Giuseppe},
	doi = {10.1215/00127094-2681605},
	fjournal = {Duke Mathematical Journal},
	issn = {0012-7094},
	journal = {Duke Math. J.},
	mrclass = {35-XX (60-XX)},
	mrnumber = {3205729},
	number = {7},
	pages = {1405--1490},
	title = {Metric measure spaces with {R}iemannian {R}icci curvature bounded from below},
	url = {http://dx.doi.org/10.1215/00127094-2681605},
	volume = {163},
	year = {2014}}

@article{AGS14,
  title={Calculus and heat flow in metric measure spaces and applications to spaces with {R}icci bounds from below},
  author={Ambrosio, Luigi and Gigli, Nicola and Savar{\'e}, Giuseppe},
  journal={Inventiones mathematicae},
  volume={195},
  number={2},
  pages={289--391},
  year={2014},
  publisher={Springer}
}

@article{sturm2005convex,
  title={Convex functionals of probability measures and nonlinear diffusions on manifolds},
  author={Sturm, Karl-Theodor},
  journal={Journal de math{\'e}matiques pures et appliqu{\'e}es},
  volume={84},
  number={2},
  pages={149--168},
  year={2005},
  publisher={Elsevier}
}

@article{ambrosio2015riemannian,
  title={Riemannian {R}icci curvature lower bounds in metric measure spaces with $\sigma$-finite measure},
  author={Ambrosio, Luigi and Gigli, Nicola and Mondino, Andrea and Rajala, Tapio},
  journal={Transactions of the American Mathematical Society},
  volume={367},
  number={7},
  pages={4661--4701},
  year={2015}
}

@article{ambrosio2013density,
  title={Density of Lipschitz functions and equivalence of weak gradients in metric measure spaces},
  author={Ambrosio, Luigi and Gigli, Nicola and Savar{\'e}, Giuseppe},
  journal={Revista Matem{\'a}tica Iberoamericana},
  volume={29},
  number={3},
  pages={969--996},
  year={2013}
}

@book{fukushima2010dirichlet,
  title={Dirichlet forms and symmetric Markov processes},
  author={Fukushima, Masatoshi and Oshima, Yoichi and Takeda, Masayoshi},
  year={2010},
  publisher={de Gruyter}
}

@article{shanmugalingam2000newtonian,
  title={Newtonian spaces: an extension of {S}obolev spaces to metric measure spaces},
  author={Shanmugalingam, Nageswari},
  journal={Revista Matem{\'a}tica Iberoamericana},
  volume={16},
  number={2},
  pages={243--279},
  year={2000}
}

@article{cheeger1999differentiability,
  title={Differentiability of Lipschitz functions on metric measure spaces},
  author={Cheeger, Jeff},
  journal={Geometric \& Functional Analysis GAFA},
  volume={9},
  number={3},
  pages={428--517},
  year={1999},
  publisher={Birkh{\"a}user Verlag}
}

@article{carbonaro_bellman_2013,
	title = {Bellman function and linear dimension-free estimates in a theorem of {Bakry}},
	volume = {265},
	issn = {0022-1236},
	url = {http://www.ams.org/mathscinet-getitem?mr=3073250},
	doi = {10.1016/j.jfa.2013.05.031},
	number = {7},
	urldate = {2016-12-12},
	journal = {J. Funct. Anal.},
	author = {Carbonaro, Andrea and Dragi{\v c}evi{\'c}, Oliver},
	year = {2013},
	mrnumber = {3073250},
	pages = {1085--1104}
}

@book{stein_topics_1970,
	series = {Annals of {Mathematics} {Studies}, {No}. 63},
	title = {Topics in harmonic analysis related to the {Littlewood}-{Paley} theory.},
	url = {http://www.ams.org/mathscinet-getitem?mr=0252961},
	urldate = {2016-07-19},
	publisher = {Princeton University Press, Princeton, N.J.; University of Tokyo Press, Tokyo},
	author = {Stein, Elias M.},
	year = {1970},
	mrnumber = {0252961}
}

@incollection{bakry_etude_1987,
	series = {Lecture {Notes} in {Mathematics}},
	title = {Etude des transformations de {Riesz} dans les vari{\'e}t{\'e}s riemanniennes {\`a} courbure de {Ricci} minor{\'e}e},
	copyright = {{\textcopyright}1987 Springer-Verlag},
	url = {http://link.springer.com/chapter/10.1007/BFb0077631},
	number = {1247},
	urldate = {2016-10-26},
	booktitle = {S{\'e}minaire de {Probabilit{\'e}s} {XXI}},
	publisher = {Springer Berlin Heidelberg},
	author = {Bakry, Dominique},
	editor = {Az{\'e}ma, Jacques and Yor, Marc and Meyer, Paul Andr{\'e}},
	year = {1987},
	doi = {10.1007/BFb0077631},
	pages = {137--172}
}

@article{junge2018noncommutative,
  title={Noncommutative Riesz transforms--dimension free bounds and Fourier multipliers},
  author={Junge, Marius and Mei, Tao and Parcet, Javier},
  journal={Journal of the European Mathematical Society},
  volume={20},
  number={3},
  pages={529--595},
  year={2018}
}

@article{junge2010noncommutative,
  title={Noncommutative Riesz transforms—a probabilistic approach},
  author={Junge, Marius and Mei, Tao},
  journal={American journal of mathematics},
  volume={132},
  number={3},
  pages={611--680},
  year={2010},
  publisher={Johns Hopkins University Press}
}

@book{oskekowski2012sharp,
  title={Sharp martingale and semimartingale inequalities},
  author={Osekowski, Adam},
  volume={72},
  year={2012},
  publisher={Springer Science \& Business Media}
}

@book{gigli_nonsmooth_2018,
	title = {Nonsmooth {Differential} {Geometry}-{An} {Approach} {Tailored} for {Spaces} with {Ricci} {Curvature} {Bounded} from {Below}},
	isbn = {978-1-4704-2765-8},
	publisher = {American Mathematical Soc.},
	author = {Gigli, Nicola},
	month = feb,
	year = {2018},
}

@article{dragicevic_linear_2012,
	title = {Linear dimension-free estimates in the embedding theorem for {Schr{\"o}dinger} operators},
	volume = {85},
	copyright = {{\textcopyright} 2012 London Mathematical Society},
	issn = {1469-7750},
	url = {https://londmathsoc.onlinelibrary.wiley.com/doi/abs/10.1112/jlms/jdr036},
	doi = {10.1112/jlms/jdr036},
	number = {1},
	urldate = {2019-05-22},
	journal = {Journal of the London Mathematical Society},
	author = {Dragi{\v c}evi{\'c}, Oliver and Volberg, Alexander},
	year = {2012},
	pages = {191--222}
}

@article{burkholder_sharp_nodate,
  title={Sharp inequalities for martingales and stochastic integrals},
  author={Burkholder, Donald L},
  journal={Ast{\'e}risque},
  volume={157},
  number={158},
  pages={75--94},
  year={1988}
}

@article{shigekawa_sobolev_1992,
	title = {Sobolev spaces over the {Wiener} space based on an {Ornstein}-{Uhlenbeck} operator},
	volume = {32},
	issn = {0023-608X},
	url = {https://projecteuclid.org/euclid.kjm/1250519405},
	doi = {10.1215/kjm/1250519405},
	number = {4},
	urldate = {2019-05-22},
	journal = {J. Math. Kyoto Univ.},
	author = {Shigekawa, Ichiro},
	year = {1992},
	mrnumber = {MR1194112},
	zmnumber = {0777.60047},
	pages = {731--748}
}

@book{stein1970topics,
  title={Topics in harmonic analysis, related to the Littlewood-Paley theory},
  author={Stein, Elias M},
  year={1970},
  publisher={Princeton University Press}
}

@article{ambrosio_lusin-type_2018,
	title = {Lusin-type approximation of {Sobolev} by {Lipschitz} functions, in {Gaussian} and {RCD}({K},$\infty$) spaces},
	volume = {339},
	issn = {0001-8708},
	url = {http://www.sciencedirect.com/science/article/pii/S0001870818303712},
	doi = {10.1016/j.aim.2018.09.033},
	urldate = {2019-05-22},
	journal = {Advances in Mathematics},
	author = {Ambrosio, Luigi and Bru{\`e}, Elia and Trevisan, Dario},
	month = dec,
	year = {2018},
	pages = {426--452}
}

@article{nazarov_hunt_1996,
	title = {The hunt for a {Bellman} function: applications to estimates for singular integral operators and to other classical problems of harmonic analysis},
	volume = {8},
	issn = {0234-0852},
	shorttitle = {The hunt for a {Bellman} function},
	url = {https://mathscinet.ams.org/mathscinet-getitem?mr=1428988},
	number = {5},
	urldate = {2019-05-22},
	journal = {Algebra i Analiz},
	author = {Nazarov, F. L. and Tre{\u i}l, S. R.},
	year = {1996},
	mrnumber = {1428988},
	pages = {32--162}
}

@article{bakry_transformation_1985,
	title = {Transformation de {Riesz} pour les semi-groupes sym{\'e}triques. {Seconde} partie : {\'e}tude sous la condition {$\Gamma_2 \geq 0$}},
	volume = {19},
	shorttitle = {Transformation de {Riesz} pour les semi-groupes sym{\'e}triques. {Seconde} partie},
	url = {https://eudml.org/doc/113508},
	urldate = {2019-06-05},
	journal = {S{\'e}minaire de probabilit{\'e}s de Strasbourg},
	author = {Bakry, Dominique},
	year = {1985},
	pages = {145--174}
}

@article{ambrosio_metric_2014,
	title = {Metric measure spaces with {Riemannian} {Ricci} curvature bounded from below},
	volume = {163},
	issn = {0012-7094},
	url = {http://arxiv.org/abs/1109.0222},
	doi = {10.1215/00127094-2681605},
	number = {7},
	urldate = {2019-06-05},
	journal = {Duke Math. J.},
	author = {Ambrosio, Luigi and Gigli, Nicola and Savar{\'e}, Giuseppe},
	month = may,
	year = {2014},
	pages = {1405--1490}
}

@article{jiang_heat_2016,
	title = {Heat {Kernel} {Bounds} on {Metric} {Measure} {Spaces} and {Some} {Applications}},
	volume = {44},
	issn = {1572-929X},
	url = {https://doi.org/10.1007/s11118-015-9521-2},
	doi = {10.1007/s11118-015-9521-2},
	number = {3},
	urldate = {2019-06-05},
	journal = {Potential Anal},
	author = {Jiang, Renjin and Li, Huaiqian and Zhang, Huichun},
	month = apr,
	year = {2016},
	pages = {601--627}
}

@article{li2019littlewood,
  title={Littlewood--{P}aley--{S}tein inequalities on {RCD(K,$\infty$)} spaces},
  author={Li, Huaiqian},
  journal={arXiv:1905.01432},
  year={2019}
}

@article{yoshida_sobolev_1992,
	title = {Sobolev spaces on a {Riemannian} manifold and their equivalence},
	volume = {32},
	issn = {0023-608X},
	url = {https://projecteuclid.org/euclid.kjm/1250519496},
	doi = {10.1215/kjm/1250519496},
	number = {3},
	urldate = {2019-06-22},
	journal = {J. Math. Kyoto Univ.},
	author = {Yoshida, Nobuo},
	year = {1992},
	mrnumber = {MR1183370},
	zmnumber = {0771.58005},
	pages = {621--654}
}

@article{meyer_transformations_1984,
	title = {Transformations de {Riesz} pour les lois gaussiennes},
	volume = {18},
	url = {https://eudml.org/doc/113479},
	urldate = {2019-06-22},
	journal = {S{\'e}minaire de probabilit{\'e}s de Strasbourg},
	author = {Meyer, Paul-Andr{\'e}},
	year = {1984},
	pages = {179--193}
}

@article{li_martingale_2008,
	title = {Martingale transforms and {Lp}-norm estimates of {Riesz} transforms on complete {Riemannian} manifolds},
	volume = {141},
	issn = {1432-2064},
	url = {https://doi.org/10.1007/s00440-007-0085-y},
	doi = {10.1007/s00440-007-0085-y},
	number = {1},
	urldate = {2019-06-22},
	journal = {Probab. Theory Relat. Fields},
	author = {Li, Xiang-Dong},
	month = may,
	year = {2008},
	pages = {247--281}
}

@article{banuelos_martingale_2013,
	title = {Martingale {Transforms} and {Their} {Projection} {Operators} on {Manifolds}},
	volume = {38},
	issn = {1572-929X},
	url = {https://doi.org/10.1007/s11118-012-9307-8},
	doi = {10.1007/s11118-012-9307-8},
	number = {4},
	urldate = {2019-06-22},
	journal = {Potential Anal},
	author = {Ba{\~n}uelos, Rodrigo and Baudoin, Fabrice},
	month = may,
	year = {2013},
	pages = {1071--1089}
}

@article{iwaniec_riesz_1995,
	title = {Riesz transforms and related singular integrals.},
	volume = {473},
	issn = {0075-4102; 1435-5345/e},
	url = {https://eudml.org/doc/153803},
	urldate = {2019-06-25},
	journal = {Journal f{\"u}r die reine und angewandte Mathematik},
	author = {Iwaniec, Tadeusz and Martin, Gaven},
	year = {1995},
	pages = {25--58}
}

@article{Sturm06I,
	author = {Sturm, Karl-Theodor},
	coden = {ACMAA8},
	fjournal = {Acta Mathematica},
	issn = {0001-5962},
	journal = {Acta Math.},
	mrclass = {53C23},
	mrnumber = {MR2237206},
	number = 1,
	pages = {65--131},
	title = {On the geometry of metric measure spaces. {I}},
	volume = 196,
	year = 2006}

@article{Sturm06II,
	author = {Sturm, Karl-Theodor},
	coden = {ACMAA8},
	fjournal = {Acta Mathematica},
	issn = {0001-5962},
	journal = {Acta Math.},
	mrclass = {53C23},
	mrnumber = {MR2237207 (2007k:53051b)},
	mrreviewer = {Juha Heinonen},
	number = {1},
	pages = {133--177},
	title = {On the geometry of metric measure spaces. {II}},
	volume = {196},
	year = {2006}}

@article{Lott-Villani09,
	author = {Lott, John and Villani, C{\'e}dric},
	coden = {ANMAAH},
	doi = {10.4007/annals.2009.169.903},
	fjournal = {Annals of Mathematics. Second Series},
	issn = {0003-486X},
	journal = {Ann. of Math. (2)},
	mrclass = {53C23 (49Q15)},
	mrnumber = {2480619 (2010i:53068)},
	mrreviewer = {Alessio Figalli},
	number = {3},
	pages = {903--991},
	title = {Ricci curvature for metric-measure spaces via optimal transport},
	url = {http://dx.doi.org/10.4007/annals.2009.169.903},
	volume = {169},
	year = {2009}}

@article{SturmVonRenesse05,
	author = {von Renesse, Max-K. and Sturm, Karl-Theodor},
	doi = {10.1002/cpa.20060},
	fjournal = {Communications on Pure and Applied Mathematics},
	issn = {0010-3640},
	journal = {Comm. Pure Appl. Math.},
	mrclass = {53C21 (35K05 58J35 58J65 60J65)},
	mrnumber = {2142879},
	mrreviewer = {C\'edric Villani},
	number = {7},
	pages = {923--940},
	title = {Transport inequalities, gradient estimates, entropy, and {R}icci curvature},
	url = {http://dx.doi.org/10.1002/cpa.20060},
	volume = {58},
	year = {2005}}

\end{document}